\documentclass[leqno,12pt]{article} %leqnoで数式番号を左側へ配置
\usepackage{amsmath, amssymb}
\usepackage{amsthm} %AMSのTheorem環境を導入
\theoremstyle{plain} % イタリック体
\newtheorem{theorem}{\indent\sc Theorem}[section] % 見出しはスモールキャップ
\newtheorem{lemma}[theorem]{\indent\sc Lemma}

\theoremstyle{definition} % ローマン体に変更

\newtheorem{remark}[theorem]{\indent\sc Remark}

\makeatletter
\def\address#1#2{\begingroup
\noindent\parbox[t]{7.8cm}{%
\small{\scshape\ignorespaces#1}\par\vskip1ex
\noindent\small{\itshape E-mail address}%
\/: #2\par\vskip4ex}\hfill%
\endgroup}%
\makeatother
\title{\uppercase{Singularities of parallel surfaces}} %論文タイトル大文字
\author{
\textsc{Toshizumi Fukui and Masaru Hasegawa} %著者名
}
\date{} %日付けは記入しない
\usepackage[dvipdfmx]{graphicx}
\usepackage{color}
\usepackage[varg]{txfonts}
\usepackage{comment}
\newcommand{\corR}[1]{#1}
\newcommand{\corC}[1]{#1}
\newcommand{\colR}[1]{#1}
\newcommand{\colB}[1]{#1}
\newcommand{\colRR}[1]{#1}
\newcommand{\BB}[1]{\boldsymbol{#1}}
\newcommand{\pd}[2]{\frac{\partial#1}{\partial#2}}

\def\rank{\operatorname{rank}}

\makeatletter
\@addtoreset{equation}{section}

\renewcommand\section{\@startsection{section}{1}{\parindent}%
                                     {-3.5ex \@plus -1ex \@minus -.2ex}%
                                     {-2.3ex \@plus -.2ex}%
                                     {\normalfont\normalsize\bfseries}}
\renewcommand\subsection{\@startsection{subsection}{2}{\parindent}%
                                     {-3.25ex\@plus -1ex \@minus -.2ex}%
                                     {-1.5ex \@plus -.2ex}%
                                     {\normalfont\normalsize\bfseries}}
\makeatother
\begin{document}

\maketitle

%%%%%%%%%%%%%%% 脚注 %%%%%%%%%%%%%%%%
\footnote{ %2000 MSC numbers
}
\footnote{ %key words and phrases
}
% \footnote{ %Thanks
% $^{*}$Partly supported by the Grant-in-Aid for Scientific Research (A),
% Japan Society for the Promotion of Science.
%科研費補助　基盤研究(A)の場合　（IV参照）　
%}
%%%%%%%%%%%%%%%%%%%%%%%%%%%%%%%%%%%%%%%%%
\begin{abstract}
 We investigate singularities of all parallel surfaces to a given
 regular surface.
 In generic context, the types of singularities of parallel surfaces
 are cuspidal edge, swallowtail, cuspidal lips, cuspidal beaks, cuspidal
 butterfly and 3-dimensional $D_4^\pm$ singularities.
 We give criteria for these singularities types in terms of differential
 geometry (Theorem~\ref{thm:main1} and \ref{cor:front}). 
\end{abstract}

\section{Introduction} %Introductionに見出し番号を付ける場合は*をとる
Classically, a wave front is the locus
of points having the same phase of vibration. 
A wave front is described by Huygens principle: 
The wave front of a propagating wave of light at any instant conforms to
the envelope of spherical wavelets emanating from every point on the
wave front at the prior instant (with the understanding that the
wavelets have the same speed as the overall wave).

It is well known that a wave front 
may have singularities at some moment. 
Singularities of wave fronts are classified in generic context   
(see \cite[p.~336]{A1}).
The local classification of bifurcations in generic one parameter
families of fronts in 3-dimensional spaces are also given in
\cite[p.~348]{A1}. 
To understand their singularities, it is important to know when the
given front is generic and when the given one parameter family is
generic.

In the differential geometric context, a wave front can be described as
the parallel surface
\[
g^t:U\to \BB{R}^3,\qquad g^{t}(u,v):=g(u,v)+t{\bf n}(u,v),
\]
of a regular surface $g:U\to\BB{R}^3$ at time $t$. 
Here $U$ is an open set of $\BB{R}^2$ and ${\bf n}$ denotes the unit
normal vector given by ${\bf n}=(g_{u}\times g_{v})/\|g_{u}\times
g_{v}\|$. 
It is well known that when $t$ is either of the principal radii of
curvature at a point of the initial surface $g$, the parallel surface
$g^t$ has a singularity at the corresponding  point (see, for example,
\cite{o1}). 
So singularities of parallel surfaces should be investigated in terms of
differential geometry of the regular map $g$.  

By Huygens principle, the wave front can be seen as the
discriminant set of the distance squared unfolding
\[
 \Phi^t: U\times\BB{R}^3\to\BB{R}, \qquad 
 (u,v,x,y,z)\mapsto -\frac12\left({\|(x,y,z)-g(u,v)\|^2-{t_0}^2}\right),
\]
where $t_0$ is a constant.
Porteous \cite{p1,p2} investigated the (Thom-Boardman) singularities of
the unfolding $(u,v,x,y,z)\mapsto\Phi^t+\frac12\|(x,y,z)\|^2$ with
$t_0=0$. 
He discovered that the notion of normal vectors, principal radii of
curvature, and umbilics correspond to $A_1$-singularities,
$A_2$-singularities, and $D_4$-singularities or worse, respectively.
Moreover, he discovered the notion of ridge points corresponding to
$A_3$-singularities or worse. 

It is now natural to ask a description of the singularity types 
of $g^t$ in terms of differential geometry, which we answer        
in this paper.  
We fix a general regular map $g$ and 
investigate singularities of $g^t$ for all $t$. 
In other words, we investigate changes of singularities due to 
time evolution of fronts generated by $g$.
To do this we need the notion of sub-parabolic point\corC{s} which is
introduced by Bruce and Wilkinson \cite{bw1} to study singularities of
folding map\corR{s}.
The main theorem (Theorem \ref{thm:main1}) states  
criteria of the singularity types of $g^t$ for all $t$ in terms of
differential geometry.   
For example, we show that, at a first order ridge point,   
$g^t$ has swallowtail singularity when it is not sub-parabolic  
where $t$ is the corresponding principal radius of curvature.  
This is enough to find \corR{a} normal form when $\Phi^t$ is an
unfolding of $A_1$, $A_2$, \corC{and} $A_3$ singularities. 
This is proved \corR{by given} a characterization 
\corR{for} the unfolding $\Phi^t$ \corR{to be} $\mathcal{K}$-versal 
in terms of differential geometry.  

We now know that $\Phi^t$ is not a $\mathcal{K}$-versal unfolding at a
sub-parabolic ridge point, \corR{a} higher order ridge, 
\corR{and an} umbilic. At these points, we are interested in the
unfolding $\Phi$ defined by
\[
\Phi: U\times\BB{R}^4\to\BB{R}, \qquad 
(u,v,x,y,z,t)\mapsto -\frac12\left(\|(x,y,z)-g(u,v)\|^2-t^2\right).
\]
Theorem \ref{thm:main1} also gives a characterization 
\corR{for} the unfolding $\Phi$ \corR{to be} $\mathcal{K}$-versal in
terms of differential geometry.  
For example, at a ridge point, we show that 
$\Phi$ is $\mathcal{K}$-versal without any other condition.  
The parallel surface is the section of discriminant set of this unfolding
with the hyperplane defined by $t=$ constant. 
\corC{For $A_4$-singularities, that is, at a second order ridge point,}
we also show (Theorem~\corC{\ref{cor:front}~(1)}) that $g^t$ has cuspidal
butterfly when \corC{it is not sub-parabolic where} $t$ is the
corresponding principal radius of curvature. 
\corC{At a sub-parabolic ridge point where $\Phi^t$ fails to be
$\mathcal{K}$-versal}, we show (Theorem~\corC{\ref{cor:front}~(2)}) the
singularities of $g^t$ are cuspidal beaks or cuspidal lips when the
corresponding CPC (constant principal curvature) line\corR{s are} Morse
singularities. For $D_4$-singularities, we also show \corR{a} similar
result (Theorem \corC{\ref{cor:front}~(3)}). These results are satisfactory in the
context of generic differential geometry.   

%The contents of this paper are follows:
%\tableofcontents

%\section{First section}
%%%%%%%%%% Theoremの例 %%%%%%%%%%
% \begin{theorem}
% In the theorem-like environment the text should be in italics, but numbers, colons and semicolons should be in roman letters.
% For instance$:$ $(a)$ $;$ $(1)$ $[1]$ $($Definition $\ref{Def}$$)$.
% So you need to enclose these by the \verb+$+ command, like
% \verb+For instance$:$+ \verb+$(a)$+ \verb+$;$+\verb+$(1)$+ \verb+$[1]$+ \verb+$($Definition $\ref{Def}$$)$+.
% \end{theorem}
%%%%%%%%%%

%%%%%%%%%% Definitionの例 %%%%%%%%%%
% \begin{definition}\label{Def}
% In the definition-like environment the text should be in roman.
% \end{definition}
%%%%%%%%%%

%%%%%%%%%% Proofの例 %%%%%%%%%%
% \begin{proof}
% The end of a proof or demonstration may be marked by an open box.
% \end{proof}
%%%%%%%%%%

\section{Preliminary from differential geometry}
We recall some differential geometric notions and their properties of
regular surfaces in Euclidean space, which we need in 
\corC{this} paper. 
We present the definitions of ridge points, sub-parabolic points and umbilics,
and their fundamental properties. 
We then discuss 
constant principal curvature (CPC) lines, 
which are the
locus of singular points of the 
parallel surface. 
We state \corR{a} characterization of these notions in terms of the
coefficients of \corR{a} 
Monge normal form of the surface.

\subsection{Fundamental forms}
Consider a surface $g$ defined by the Monge form: 
\begin{align}
 \label{eq:monge}
 g(u,v)=\left(u,v,f(u,v)\right),\qquad 
 f(u,v)=\frac{1}{2}(k_1u^2+k_2v^2)
 +\sum\limits_{i+j\geq 3}\frac{1}{i!j!}a_{ij}u^{i}v^{j}.
\end{align}
The coefficient\corR{s} of the first fundamental form \corR{are} given by
\[
E=\langle g_u,g_u\rangle=1+{f_u}^2,\qquad
F=\langle g_u,g_v\rangle={f_u}{f_v},\qquad
G=\langle g_v,g_v\rangle=1+{f_v}^2. 
\]
Here subscripts denote\corR{s} partial derivatives \corR{and}
$\langle\,,\,\rangle$ denote\corR{s} the Euclidean inner product of
$\BB{R}^{3}$. The unit normal vector is given by 
\[
{\bf n}=\frac1{\sqrt{1+{f_u}^2+{f_v}^2}}(-f_u,-f_v,1).
\]
The coefficient\corR{s} of the second fundamental form \corR{are} given by
\colB{\begin{align*}
 {\small
 L=\langle g_{uu},{\bf
 n}\rangle=\frac{f_{uu}}{\sqrt{1+{f_u}^2+{f_v}^2}},\ \   
 M=\langle g_{uv},{\bf
 n}\rangle=\frac{f_{uv}}{\sqrt{1+{f_u}^2+{f_v}^2}},\ \ 
 N=\langle g_{vv},{\bf
 n}\rangle=\frac{f_{vv}}{\sqrt{1+{f_u}^2+{f_v}^2}}.
 }
\end{align*}}
% We consider the matrices of the 
% first fundamental form and the second fundamental form:  
% \[
%   \mathrm{I}=
%  \begin{pmatrix}
%   E& F\\
%   F& G
%  \end{pmatrix},\qquad
%  \II=
%  \begin{pmatrix}
%   L& M\\
%   M& N
%  \end{pmatrix}.
% \]
% \corR{The} parabolic line is defined by
% \[
%  \begin{vmatrix}
%   f_{uu}&f_{uv}\\
%   f_{uv}&f_{vv}
%  \end{vmatrix}
%  =k_1k_2+(a_{12}k_1+a_{30}k_2)u+(a_{03}k_1+a_{12}k_2)v+\textrm{h.o.t.}=0.
% \]

\subsection{Principal curvatures}
We say that $\kappa$ is a \corR{{\it principal curvature}} 
if there is \corR{a} non-zero vector $(\xi,\zeta)$ such that 
\colB{\begin{equation}
 \label{eq:PriVec}
 \begin{pmatrix}
  L&M\\
  M&N
 \end{pmatrix}
 \begin{pmatrix}
  \xi\\\zeta
 \end{pmatrix}
 =\kappa
 \begin{pmatrix}
  E&F\\
  F&G
 \end{pmatrix}
 \begin{pmatrix}
  \xi\\\zeta
 \end{pmatrix}.
\end{equation}}
% or, equivalently,
% \begin{align*}
%  \frac1{EG-F^2}
%  \begin{pmatrix}
%   G&-F\\
%   -F&E
%  \end{pmatrix}
%  \begin{pmatrix}
%   L&M\\
%   M&N
%  \end{pmatrix}
%  \begin{pmatrix}
%   \xi\\\zeta
%  \end{pmatrix}
%  =
%  \kappa
%  \begin{pmatrix}
%   \xi\\\zeta
%  \end{pmatrix}.
%\end{align*}
This is rewritten as 
$$
\frac1{(1+{f_u}^2+{f_v}^2)^{3/2}
}\begin{pmatrix}
1+{f_v}^2&-f_uf_v\\
-f_uf_v&1+{f_u}^2
 \end{pmatrix}
\begin{pmatrix}
f_{uu}&f_{uv}\\
f_{uv}&f_{vv}
\end{pmatrix}
\begin{pmatrix}
\xi\\\zeta
\end{pmatrix}
=
\kappa
\begin{pmatrix}
\xi\\\zeta
\end{pmatrix}.
$$

The eigenvector $(\xi_i,\zeta_i)$ $(i=1,2)$ of the equation
\eqref{eq:PriVec} corresponding to the eigenvalue $\kappa_{i}$ gives the 
principal vector ${\bf v}_i$. We can choose them so that the tangent
vectors $\xi_ig_u+\zeta_ig_v$ are of \corR{the} unit length.
At a point on the surface where two principal curvatures are distinct,
there are two principal \corC{vectors} and these \corC{vectors} are
mutually orthogonal. These principal \corC{vectors} are often colored
(blue or red) to distinguish between the two \corC{vectors}. We assume
that ${\bf v}_1$ is the blue principal \corC{vector} and ${\bf v}_2$ is
the red principal \corC{vector}.

\colB{Suppose that $k_1\ne k_2$, ${\bf v}_1=(1,0)$, and ${\bf
v}_2=(0,1)$.}
\colB{The principal curvature $\kappa_1$ is expressed as
 \begin{align}
  \label{eq:k1}
  \begin{split}
   \kappa_1(u,v)
   &=k_1+a_{30}u+a_{21}v
   +\frac{1}{2(k_1-k_2)}\{[2{a_{21}}^2+(a_{40}-3{k_1}^3)(k_1-k_2)]u^2\\ 
   &\quad+2[2a_{21}a_{12}+a_{31}(k_1-k_2)]uv
   +[2{a_{12}}^2+(a_{22}-k_1{k_2}^2)(k_1-k_2)]v^2\}+O(u,v)^3,
  \end{split}
 \end{align}
 and we have
 \begin{equation}
  \label{eq:3-jet_k1}
  \frac{\partial^3\kappa_1}{\partial u^3}(0,0)
   =\frac{6{a_{21}}^2(-a_{30}+a_{12})+6a_{21}a_{31}(k_1-k_2)+(a_{50}-18a_{30}{k_1}^2)(k_1-k_2)^2}
   {6(k_1-k_2)^2}.
 \end{equation}}
 \colB{It follows form \eqref{eq:PriVec} that there 
 is a real number $\mu\ne0$ 
 such that $(\xi_1,\zeta_1)=\mu(N-\kappa_1G,-M+\kappa_1F)$.
%  \[
%   (\xi_1,\zeta_1)=\mu(N-\kappa_1G,-M+\kappa_1F).
%  \]
 Selection of %the vector
 $(\xi_1,\zeta_1)$ in order for the tangent
 vector $\xi_1g_u+\zeta_1g_v$ to be of \corC{the} unit length shows that  
 %the principal vector
 ${\bf v}_1$ is expressed as
 \begin{align}
  \label{eq:V1}
  {\bf v}_1(u,v)=&
  \Biggl(1+O(u,v)^2\Biggr)\pd{}{u}
  +\Biggl(\frac{1}{k_1-k_2}(a_{21}u+a_{12}v)+O(u,v)^2\Biggr)\pd{}{v},
 \end{align}
 and that
 \begin{equation}
  \label{eq:2-jet_v1}
  \frac{\partial^2\zeta_1}{\partial
  u^2}(0,0)=\frac{2a_{21}(a_{12}-a_{30})+a_{31}(k_1-k_2)}{2(k_1-k_2)^2}.
 \end{equation}}
 \colB{Since ${\bf v}_1$ and ${\bf v}_2$ are orthogonal,
 it follows from \eqref{eq:V1} that ${\bf v}_2$ is expressed as
 \begin{align}
  \label{eq:V2}
  {\bf v}_2(u,v)=&
  \Biggl(\frac{1}{k_2-k_1}(a_{21}u+a_{12}v)+O(u,v)^{2}\Biggr)\corC{\pd{}{u}}
  +\Biggl(1+O(u,v)^{2}\Biggr)\corC{\pd{}{v}}.
 \end{align}}

 If two principal curvatures are equal at a point on the surface, we call
 such a point an umbilic.
 At an umbilic every direction through the umbilic is principal and
 the umbilic is \corR{an} isolated singularity of the direction
 \corR{field}.

 If only one principal curvature is zero, such a point is called a parabolic
 point.
 If both principal curvatures are zero, such a point is called a
 flat umbilic or a planer point.
 
 We can consider the focal surface. %\colR{Away form umbilics,}
 \colRR{Except for the umbilics,} the focal surface
 consists \corR{of} two sheets, the blue and red sheets given by $g+{\bf
 n}/\kappa_1$ and $g+{\bf n}/\kappa_2$, respectively. 
 The two sheet\corR{s} come together at umbilics.
 We note that at parabolic points only one of the two sheet\corR{s} exits, and at
 flat umbilics the common focal point lies at infinity.
 
 The focal surface might have a singular point where the same colored
 principal curvature has an extreme value along the same colored line of
 curvature.
 Such a point on $g$ is called a \colR{{\it ridge}} point and on \colR{the} focal surface a
 \colR{{\it rib}}.
 Ridges were first studied in details by Porteous \cite{p1}.
 
 The locus of points where the principal curvature has extreme value along
 the other colored line of curvature is also of importance.
 This locus is called a sub-parabolic line.
 The sub-parabolic line \colR{was} studied in details by Bruce and
 Wilkinson \cite{bw1} in terms of folding maps.
 The sub-parabolic line is also the locus of points on the surface whose
 image is the parabolic line on the same colored sheet of the focal surface.
 In \corR{\cite{m2}} the sub-parabolic line appear as the locus of points where
 the other colored line of curvature has the geodesic inflections.
 
\subsection{Ridge points and sub-parabolic points}
\label{sec:Ridge}
Let $g(p)$ be \corC{not an umbilic} 
of a \corC{regular} surface $g$. %,  
% with principal \corC{vectors} ${\bf v}_1$ (`blue'), ${\bf v}_2$ (`red')
% \corC{corresponding} principal curvature $\kappa_{1}$, $\kappa_{2}$.
We say that the point $g(p)$ is a \corR{{\it ridge point}} relative to
${\bf v}_i$ (`blue ridge point' for $i=1$, `red'
for $i=2$) if ${\bf v}_i\kappa_i(p)=0$, where ${\bf v}_i\kappa_i$
is the directional derivative of $\kappa_i$ in 
${\bf v}_i$.
Moreover, $g(p)$ is a \corR{$k$-{\it th order ridge point}}
\corC{relative to ${\bf v}_i$} if \colB{${\bf
v}_i^{(m)}\kappa_i(p)=0$ $(1\leq m\leq k)$ and ${\bf v}_i^{(k+1)}\kappa_i(p)\neq 0$,}
% \[
% {\bf v}_i^{(m)}\kappa_i(p)=0\quad(1\leq m\leq k)\qquad\corC{\textrm{and}}\qquad
% {\bf v}_i^{(k+1)}\kappa_i(p)\neq 0, 
% \]
where ${\bf v}_i^{(k)}\kappa_i$ is the $k$-times directional
derivative of $\kappa_i$ in ${\bf v}_i$.
The set of ridge points is called a \corR{{\it ridge line}} or \corC{{\it ridges}}.

\colB{We turn to sub-parabolic points.
A point $g(p)$ which is not an umbilic is a {\it
sub-parabolic point} relative to ${\bf v}_i$ (`blue sub-parabolic
point' for $i=1$, `red' for $i=2$) if ${\bf v}_i\kappa_j(p)=0$~$(i\neq
j)$. The set of sub-parabolic points is called a {\it
sub-parabolic line}.}

\colB{Let $g$ be given in Monge form as in \eqref{eq:monge}, and let
$k_1\ne k_2$.
From \eqref{eq:k1} through \eqref{eq:V2}, we obtain the following lemmas.}

\begin{lemma}
\label{lem:ridge}
%  Suppose that a surface $g$ is given in Monge form as in
%  \eqref{eq:monge}, and that the origin is not an umbilic. 
 \begin{enumerate}
  \vspace{-8pt}
  \setlength{\parskip}{0cm} % 段落間
  \setlength{\itemsep}{0cm} % 項目間
   \item The origin is a first order blue ridge point if and only if
         \[
         a_{30}=0\quad\corC{\mathrm{and}}\quad
         3{a_{21}}^2+(a_{40}-3{k_1}^3)(k_1-k_2)\neq 0. 
         \]
         
  \item The origin is a second order blue ridge point if and only if  
         \begin{align*}
          &a_{30}=3{a_{21}}^2+(a_{40}-3{k_1}^3)(k_1-k_2)=0\quad \mathrm{and}\\
          &15{a_{21}}^2a_{12}+10a_{21}a_{31}(k_1-k_2)+a_{50}(k_1-k_2)^2\neq
          0.
         \end{align*}
 \end{enumerate}
\end{lemma}
\begin{lemma}
 \label{lem:subpara}
 \colB{The origin is a red sub-parabolic point if and only if $a_{21}=0$.}
\end{lemma}

\colB{From \eqref{eq:k1}, \eqref{eq:V1}, and \eqref{eq:V2}, it follows
that the equation of the blue ridge line through the origin is expressed as 
\begin{equation}
 \label{eq:rline}
 [3{a_{21}}^2+(a_{40}-3{k_1}^{3})(k_1-k_2)]u+[3a_{21}a_{12}+a_{31}(k_1-k_2)]v+\cdots =0.
\end{equation}
and that the red sub-parabolic line through the origin is expressed as
\begin{equation}
 \label{eq:spline}
  a_{31}(k_1-k_2)u+[a_{12}(2a_{12}-a_{30})+(a_{22}-k_1{k_2}^{2})(k_1-k_2)]v+\cdots
  =0.
\end{equation}
 Equation \eqref{eq:rline} implies the following lemma.}
\begin{lemma}
 \label{lem:ridgeline}
 Suppose that %a surface $g$ is given 
%  in Monge form as in \eqref{eq:monge}, and that
 the origin is a blue ridge point.
%Then:
%\begin{enumerate}
% \item the blue ridge line through the origin fails to be smooth if and
%	only if 
%	\begin{align*}
%	 3{a_{21}}^{2}+(a_{40}-3{k_1}^{3})(k_1-k_2)=0,\quad
%	 3a_{21}a_{12}+a_{31}(k_1-k_2)=0;
%	\end{align*}
%  \item the tangent to the blue ridge line
%	through the origin is along the blue principal direction
%	${\bf v}_1=(1,0)$ if and only if
%	\begin{align*}
%	 3{a_{21}}^2+(a_{40}-3{k_1}^{3})(k_1-k_2)=0,\quad	 3a_{21}a_{12}+a_{31}(k_1-k_2)\neq 0.
%	\end{align*}
 %\end{enumerate}
 Then the blue ridge line % through the origin fails to be smooth at the
 % origin  
 \colB{has a singular point at the origin} if and only if 
 \begin{align*}
  3{a_{21}}^2+(a_{40}-3{k_1}^3)(k_1-k_2)=%0,\quad
  3a_{21}a_{12}+a_{31}(k_1-k_2)=0.
 \end{align*}
 \end{lemma}

\subsection{Umbilics}
\label{sec:umbilics}
{\it Umbilics} of a regular surface are points where the two principal
curvatures coincide.  
% At these points the principal direction field is singular and the lines
% of curvature fail to cross at right angle.
The classification of generic umbilics is due to Darboux \cite{d}.
He gave three configurations of the lines of curvature. 
The three configurations were given the names lemon, star, and
monstar by Berry and Hannay \cite{BH}.
Their classification was provided by Gutierrez and Sotomayor \cite{gs}.

Suppose that the origin is an umbilic of a surface $g$, and that $g$ is
given in Monge form
\begin{align}
 \label{eq:umbmonge}
 g(u,v)=(u,v,f(u,v)),\quad
 f(u,v)=\dfrac{k}{2}(u^{2}+v^{2})+\sum\limits_{i+j\geq
 3}\frac{1}{i!j!}a_{ij}u^{i}v^{j},
\end{align}
where $k$ is the common value for the principal curvatures at the
origin.

At %an
\colB{the} umbilic the cubic part $f_3$ of $f$ in (\ref{eq:umbmonge})
determines its type. 
%An
\colB{The} umbilic of %the surface
$g$ is said to be {\it elliptic} or {\it
hyperbolic} if $f_3$ has three real roots or one real root,
respectively.
Moreover, \colB{the} umbilic is said to be {\it right-angled} if the root
directions of the quadratic form which is the determinant of the Hessian
matrix of $f_3$ are mutually orthogonal with respect to the standard
scalar product on $\BB{R}^2$.
Such an umbilic necessarily is \corC{a} hyperbolic umbilic.

\corC{We shall present the conditions for types of umbilics in terms of
the coefficients of the Monge form.}
We %consider the following determinants:
\corC{set}
\begin{align*}
 \Gamma :=
 \colR{\begin{array}{|cccc|}
  a_{30} & 2a_{21} & a_{12} & 0 \\
  0 & a_{30} & 2a_{21} & a_{12} \\
  a_{21} & 2a_{12} & a_{03} & 0 \\
  0 & a_{21} & 2a_{12} & a_{03} \\
 \end{array}
 \ ,}\quad\corC{\mathrm{and}}\quad
 \Gamma^{\prime} :=
 \colR{\begin{array}{|ccc|}
  1 & 0 & 1 \\
  a_{30} & a_{21} & a_{12} \\
  a_{21} & a_{12} & a_{03} \\
 \end{array}
 \ .}
\end{align*}
\corC{The discriminant of $f_3$ is given by $-\Gamma$. Hence,} the
origin is an elliptic umbilic %, parabolic umbilic, 
or hyperbolic umbilic %when the following holds:
\corC{if and only if $\Gamma<0$ %, $\Gamma=0$,
or $\Gamma>0$, respectively.}
%We characterize these types of umbilics by the coefficient\corC{s} of the Monge
%form \eqref{eq:umbmonge}.
% \begin{itemize}
%  \item Elliptic umbilic\,:\,$\Gamma<0$,
%        %If $\Gamma>0$, the origin is an elliptic umbilic.
%  \item Parabolic umbilic\,:\,$\Gamma=0$,
%        %If $\Gamma=0$, the origin is a parabolic umbilic.
%  \item Hyperbolic umbilic\,:\,$\Gamma>0$.
%        %If $\Gamma<0$, the origin is a hyperbolic umbilic.
%  %\item If $\Gamma^{\prime}=0$, the origin is a right-angled umbilic.
% \end{itemize}
Moreover, \corC{the determinant of the Hessian matrix of $f_3$ is given by
\[
 -36[({a_{21}}^2-a_{30}a_{12})u^2+(a_{21}a_{12}-a_{30}a_{03})uv+({a_{12}}^2-a_{21}a_{03})v^2].
\]}
\corC{It follows that} the origin is a right-angled umbilic %when 
\corC{if and only if} $\Gamma^\prime=0$.

It is shown in %I.~Porteous
\cite%[\corC{p.~191}]
{p2} that there is one ridge line passing
through a hyperbolic umbilic and three ridge lines passing through an
elliptic umbilic. \corC{It is also shown in \cite%[p.~191 and p.~202]
{p2} that ridge lines change %its
their color as they pass through a generic
umbilic. % except the right-angled umbilic.
%The ridge line pass through a right-angled umbilic
%dose not change its color there.
}

%Bruce and Willinson \cite{bw1} show
% sub-parabolic lines behave similarly to ridge lines at umbilic and the
% behavior of sub-parabolic lines is closely related to the lines of
% curvature.
% In general, there are one or three sub-parabolic lines passing through an
% umbilic.
\corC{It is known that when there is one direction for lines of curvature
at an umbilic, there is one sub-parabolic line through the umbilic in the
same direction, while, when there are three directions for
lines of curvature at an umbilic, there are three sub-parabolic lines
through the umbilic in the same three directions \cite{bw1,m2}.}
% We note that ridge lines change color as they pass through the umbilic
% point.
% The color change of the
% sub-parabolic line\corC{s} is similar to the ridge line. 

\subsection{Constant principal curvature lines}
\corC{We set
\[
 \Sigma_c:=\{(u,v)\in U\ ;\ \kappa_i(u,v)=c\ \text{for some}\ i\}.
\]
We call $\Sigma_c$ the {\it constant principal curvature \colR{$($CPC$)$}
line with the value of $c$}. %The CPC
% line $\Sigma_c$ is formed by two components $\kappa_1(u,v)=c$ and
% $\kappa_2(u,v)=c$. We distinguish the two components by assigning colors
% (blue and red).
There %exist
are two CPC lines $\Sigma_{\kappa_1(p)}$ (colored by
blue) and $\Sigma_{\kappa_2(p)}$ (colored by red) locally through a
%point $p$ where $g(p)$ is not an umbilic.
non-umbilical point $g(p)$.
We recall that a point $p\in U$ is a singular point of the
parallel surface $g^t$ at distance $t$ if and only if
$t=1/\kappa_i(p)$ for some $i$. This means that the set of
singular points of $g^t$ is the CPC line $\Sigma_{\kappa_i(p)}$.}
% , \corC{and that the constant principal curvature (CPC) line $\kappa_i(u,v)=1/t$ is the set
% of singular points of the parallel surface $g^t$ at distance $t$.}
% \corC{Similar to ridge and sub-parabolic lines, we distinguish CPC lines
% by assigning colors. We assume that CPC lines
% $\kappa_1(u,v)=k$ are {\it blue CPC lines} and CPC lines
% $\kappa_2(u,v)=k$ are {\it red CPC lines}, where $k$ is a constant.}
% As we shall see in detail later, 
% singularities of %the
% \corC{a} parallel surface are determined by differential geometric
% properties of the initial surface, namely, a ridge point, a
% sub-parabolic point, and an umbilic. For this reason it is important to
% see the configuration of %constant principal curvature (CPC)
% CPC lines, %$\kappa_i(u,v)=1/t$, which form the set of singular points of $g^t$,
% ridge lines, and sub-parabolic lines. 

Firstly, we %consider
\corC{investigate} %the configuration of
the CPC line\corC{s} away form umbilics.
Suppose that a surface $g$ is given %locally
in Monge form as in \eqref{eq:monge}.
From \eqref{eq:k1}, %it turns out that the CPC line
$\kappa_1(u,v)=k_1$ is %written in the expansion form
\corC{expressed by the equation}
\begin{align}
 \label{eq:cK1}
 \begin{split}
  0&=a_{30}u+a_{21}v +\frac{1}{2(k_1-k_2)}\{
  [2{a_{21}}^2+(a_{40}-3{k_1}^3)(k_1-k_2)]u^2\\
  &\quad+2[2a_{21}a_{12}+a_{31}(k_1-k_2)]uv
  +[2{a_{12}}^2+(a_{22}-k_1{k_2}^2)(k_1-k_2)]v^2\}+\cdots.
 \end{split}
\end{align}
% \corC{Hence, the CPC line $\Sigma_{k_1}$ is locally given by the equation
% \eqref{eq:cK1} at the origin.} 
\colB{This equation}
%\corC{The equation \eqref{eq:cK1}}
shows that the CPC line
%$\kappa_1(u,v)=k_1$
\corC{$\Sigma_{k_1}$} is singular at the origin if and only if
$a_{30}=a_{21}=0$, that is, the origin is a blue ridge point and %a
red sub-parabolic point (Lemma \ref{lem:ridge} and \ref{lem:subpara}). 

\begin{lemma}
\label{lem:CPCline}
 Suppose that the origin is a blue ridge point which is not a red
 sub-parabolic point.
 \colB{The CPC line $\Sigma_{k_1}$ is transverse $($resp. tangential$)$ to the
 blue ridge line at the origin if and only if the order of the ridge is
 one $($resp. more than one$)$.}
% \begin{enumerate}
%   \vspace{-8pt}
%   \setlength{\parskip}{0cm} % 段落間
%   \setlength{\itemsep}{0cm} % 項目間
%  \item The CPC line $\Sigma_{k_1}$ is transverse to the blue ridge line
%        at the origin if and only if the order of the ridge is one.  
%  \item the CPC line $\Sigma_{k_1}$ is tangential to the blue ridge line
%        at the origin if and only if the order of the ridge is more than
%        one. 
% \end{enumerate}
\end{lemma}
\begin{proof}
 It follows from \eqref{eq:rline} and \eqref{eq:cK1} that the CPC line
 $\Sigma_{k_1}$ is transverse \colB{(resp. tangentail)} to the blue
 ridge line at the origin if and only if
%  \[
%   3{a_{21}}^2+(a_{40}-3{k_1}^3)(k_1-k_2)\ne0.
%  \]
%  On the other hand, both lines are tangential at the origin  if and only if 
%  \[
%   3{a_{21}}^2+(a_{40}-3{k_1}^3)(k_1-k_2)=0.
%  \]
 \[
 3{a_{21}}^2+(a_{40}-3{k_1}^3)(k_1-k_2)\ne0\quad
 \colB{(\text{resp.}\ 3{a_{21}}^2+(a_{40}-3{k_1}^3)(k_1-k_2)=0).}
 \]
 Hence, the \colR{assertion} of the lemma follows from Lemma
 \ref{lem:ridge}. 
\end{proof}

\begin{lemma}
 Suppose that the origin is a blue ridge point and red sub-parabolic point.
 Then the CPC line $\Sigma_{k_1}$ is locally either an isolated point or
 \colR{the union of} two intersecting smooth curves at the origin, if
 the blue ridge line crosses the red sub-parabolic line at the origin.  
\end{lemma}
\begin{proof}
 First we remark that
 \begin{align*}
  \frac{\partial \kappa_1}{\partial
  u}(0,0)=a_{30}=0\quad\corC{\mathrm{and}}\quad \frac{\partial
  \kappa_1}{\partial v}(0,0)=a_{21}=0. 
 \end{align*}
 The equations of the blue ridge line \eqref{eq:rline} and the red
 sub-parabolic line \eqref{eq:spline} \colR{are reduce to}
 \[
 (a_{40}-3{k_1}^3)(k_1-k_2)u+a_{31}(k_1-k_2)v+\cdots =0
 \]
 and
 \[
 a_{31}(k_1-k_2)u+[2{a_{12}}^2+(a_{22}-k_1{k_2}^2)(k_1-k_2)]v+\cdots
 =0,
 \]
 respectively.
 From these equations, the blue ridge line crosses the red sub-parabolic
 line at the origin if and only if \colB{$A\ne0$, where}
%  \[
%  (a_{40}-3{k_1}^3)(k_1-k_2)[2{a_{12}}^2
%  +(a_{22}-k_1{k_2}^2)(k_1-k_2)]-{a_{31}}^2(k_1-k_2)^2\ne0. 
%  \]
\colB{\[
  A=(a_{40}-3{k_1}^3)(k_1-k_2)[2{a_{12}}^2+(a_{22}-k_1{k_2}^2)(k_1-k_2)]-{a_{31}}^2(k_1-k_2)^2. 
 \]}
 In addition, from \eqref{eq:k1}, the determinant of the Hessian matrix
 of $\kappa_1$ at $(0,0)$ is given by \colB{$A$.}
%  \[
%  (a_{40}-3{k_1}^3)(k_1-k_2)[2{a_{12}}^2+(a_{22}-k_1{k_2}^2)(k_1-k_2)]-{a_{31}}^2(k_1-k_2)^2.
%  \]
 By the Morse lemma (\corC{see, for example,} \cite{b3}), we complete the proof.
\end{proof}

Secondly, we investigate the CPC line near an umbilic.
\begin{theorem}
 \label{thm:cpc_at_umbilic}
 \begin{enumerate}
   \vspace{-8pt}
   \setlength{\parskip}{0cm} % 段落間
   \setlength{\itemsep}{0cm} % 項目間
  \item The CPC line $\Sigma_k$ is locally an isolated point at
        %the
        \colB{an} elliptic umbilic, where $k$ is the common value for
        the principal curvatures at the umbilic. 
  \item The CPC line $\Sigma_k$ is locally two intersecting smooth
        curves at a hyperbolic umbilic. The locally two curves change
        their color as they pass through the hyperbolic umbilic. 
 \end{enumerate}
\end{theorem}
\begin{proof}
 We suppose that the origin is an umbilic of \corC{a surface} $g$, and
 that the surface $g$ is given in Monge form as in \eqref{eq:umbmonge}. 
 The principal curvatures are the roots of the quadric equation
 \[
 (EG-F^2)\corC{\kappa}^2-(EN-2FM+GL)\corC{\kappa}+(LN-M^2)=0.
 \]
 Replacing $\kappa$ by $k$ which is the common value for the
 principal curvatures at the origin, we can express the equation in the
 form 
 \begin{align}
  \label{eq:CPCLineAtUmbilic}
  (a_{30}a_{12}-{a_{21}}^2)u^2+(a_{30}a_{03}-a_{21}a_{12})uv
  +(a_{21}a_{03}-{a_{12}}^2)v^2+\cdots=0.
 \end{align}
 The locus of this equation is %the CPC line
 $\Sigma_k$.
 We denote the quadric part of \eqref{eq:CPCLineAtUmbilic} by $\alpha
 u^2+2\beta uv+\gamma v^2$. 
 Then we have $\beta^2-\alpha \gamma=\Gamma/4$,
 where $\Gamma$ is as in %Section
 \colB{Subsection}~\ref{sec:umbilics}.
 Hence, %the CPC line
 $\Sigma_k$ at an umbilic is locally either an
 isolated point if %$\Gamma<0$ (i.e.,
 the origin is an elliptic umbilic %)
 or two smooth intersecting curves if %$\Gamma>0$ (i.e.,
 the origin is a hyperbolic umbilic %)
 . 
 
 We investigate the case of hyperbolic umbilics in detail.
 For a hyperbolic umbilic, we may assume that $g$ is locally given in
 the form 
 \begin{align}
  \label{eq:HyperUmbiForm}
  g(u,v)=(u,v,f(u,v)),\quad f(u,v)
  =\frac k2(u^2+v^2)+\frac P6u(u^2+2Quv+Rv^2)+\cdots
 \end{align}
 for some $P$, $Q$, and $R$ with $P\ne0$ and $Q^2-R<0$.
 Then % the principal curvatures $\kappa_1$ (maximum
%  curvature), $\kappa_2$ (minimum curvature)
 \colB{$\kappa_1$ and $\kappa_2$ $(\kappa_1\geq\kappa_2)$}
 are expressed as
%  \begin{align}
%   \label{eq:PCsAtHyperUmbi}
%   \begin{split}
%    \kappa_1(u,v)&=k+\dfrac16\Bigl(P[(R+3)u+2Qu]\Bigr.\\
%    &\quad\Bigl.+|P|\sqrt{[16Q^2+(R-3)^2]u^2+12Q(R+1)uv+4(Q^2+R^2)v^2}\Bigr)+\cdots,\\
%    \kappa_2(u,v)&=k+\dfrac16\Bigl(P[(R+3)u+2Qu]\Bigr.\\
%    &\Bigl.\quad-|P|\sqrt{[16Q^2+(R-3)^2]u^2+12Q(R+1)uv+4(Q^2+R^2)v^2}\Bigr)+\cdots.
%   \end{split}
%  \end{align}
 \colB{
 {\small 
 \[
 \kappa_i(u,v)
 =k+\dfrac16\Bigl(P[(R+3)u+2Qu]
 +\varepsilon|P|\sqrt{[16Q^2+(R-3)^2]u^2+12Q(R+1)uv+4(Q^2+R^2)v^2}\Bigr)+\cdots,
 \]}}
 %where $\varepsilon=\pm1$.}
 \colRR{where $\varepsilon=1$ for $i=1$ and $-1$ for $i=2$.}
 Therefore, the locally two smooth curves change their color as they
 \colR{go} through the hyperbolic umbilic. 
\end{proof}

\begin{remark}
 \begin{enumerate}
    \vspace{-8pt}
    \setlength{\parskip}{0cm} % 段落間
    \setlength{\itemsep}{0cm} % 項目間
  \item A simple calculation % we have
        gives $\Gamma^\prime=\alpha+\gamma$, where $\Gamma^\prime$
        is as in %Section
        \colB{Subsection}~\ref{sec:umbilics}.
        It follows that the tangents to the locally two smooth curves of
        %the CPC line
        \colB{$\Sigma_k$} through the right-angled umbilic are  mutually
        orthogonal.
%         We note that the right-angled umbilic 
%         necessarily is a hyperbolic umbilic.
  \item Equation \eqref{eq:CPCLineAtUmbilic} shows that %the CPC line
        \colB{$\Sigma_k$} is approximated by a conic near the origin
        when the origin is %not a parabolic
        \colB{an elliptic or hyperbolic} umbilic.
 \end{enumerate}
\end{remark}

\corC{Finally,} We %consider
\corC{investigate} bifurcations of the CPC lines %$\kappa_1(u,v)=k$ and
%$\kappa_2(u,v)=k$
at \corC{an} umbilic. %, where $k$ is the common value
%for the principal curvature at the umbilic point.
We start with the case of an elliptic umbilic.
There are three ridge lines through the elliptic umbilic.
The bifurcation of the CPC lines at the elliptic
umbilic is shown in Figure~\ref{fig:CPCLineAtUmbilic}~(i), \corC{(ii)
(cf.~\colR{\cite[Figure~2]{b1}})}.  
We now turn to the case of a hyperbolic umbilic.
We %consider the form \eqref{eq:HyperUmbiForm}.
\corC{may assume that the surface given in} the from
\eqref{eq:HyperUmbiForm}. There is one ridge line through the hyperbolic 
umbilic. %and it is $u$-axis except the origin.
\corC{Calculations show that the ridge line is tangent to $2Qu+Rv=0$ at
the origin \colR{(cf.~\cite[part (iii) of the corollary of
Theorem~11.10]{p2}),} and that the
locally two smooth curves of %the CPC line
$\Sigma_k$ are tangent to
$[QR\pm\sqrt{R^2(-Q^2+R)}]u+R^2v=0$.}
Thus it follows %from \eqref{eq:PCsAtHyperUmbi}
that the bifurcation of the CPC lines at the hyperbolic umbilic is given
%by
\corC{in} Figure~\ref{fig:CPCLineAtUmbilic}~%(ii)
\corC{(iii) \colR{through} (v) (cf.~\colR{\cite[Figure~2]{b1}}), in the generic context}.

\begin{figure}[htbp]
 \begin{center}
  \input{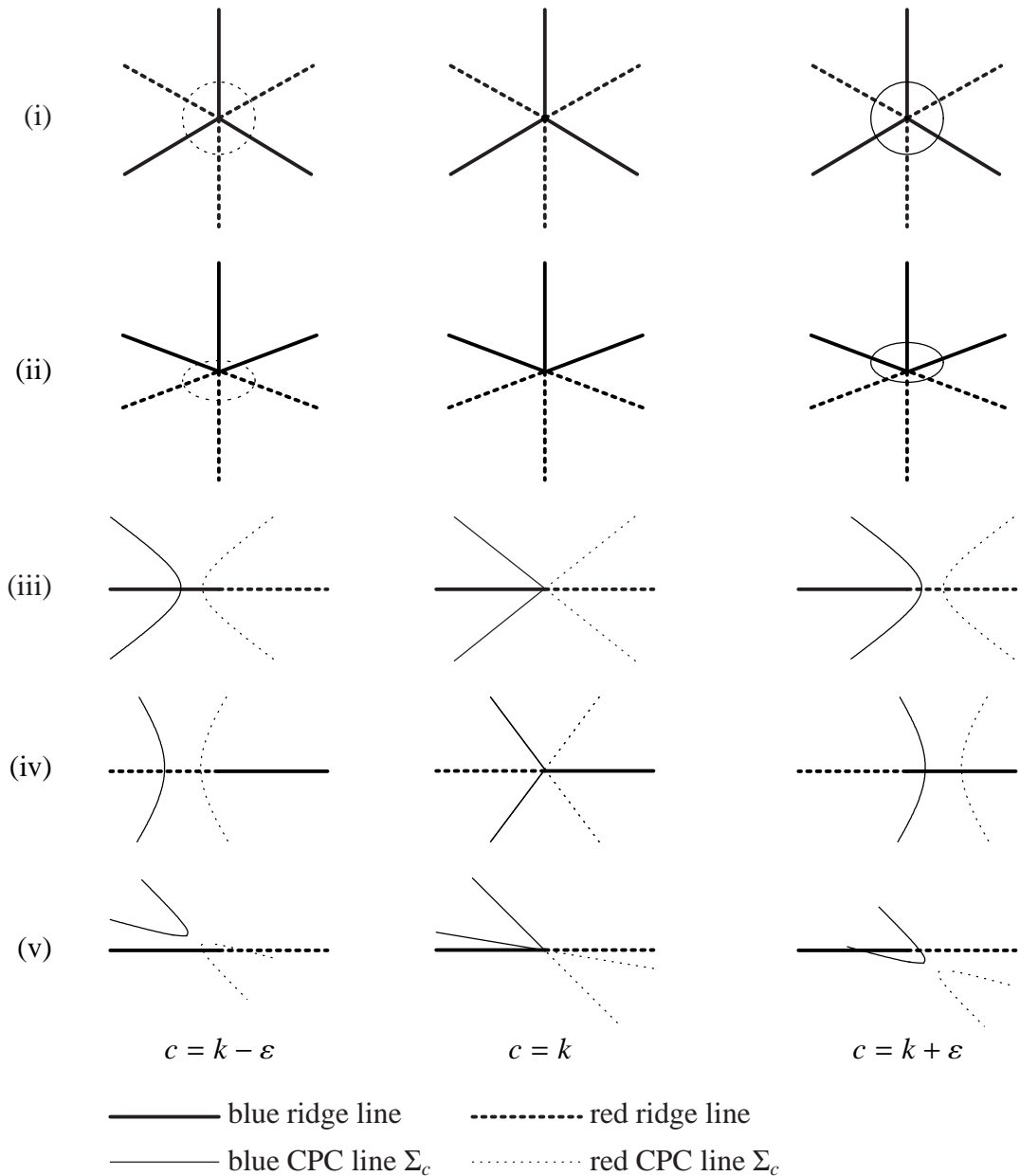}
  \caption
  {
  Bifurcations of %constant principal curvature
  \corC{the CPC} lines %at an umbilic elliptic
  \corC{near an elliptic umbilic (i) and (ii), and  a hyperbolic umbilic
  (iii) \colR{through} (v)}
  % umbilic (i) and a hyperbolic umbilic (ii)
  , where $\varepsilon$ is a
  small positive number.
  }
  \label{fig:CPCLineAtUmbilic}
 \end{center}
\end{figure} 

As shown %from
\corC{in} Figure~\ref{fig:CPCLineAtUmbilic},
there are three intersection points of the CPC line and the
\corC{same colored} ridge line %relative to the same principal \corC{vector}
near \corC{an} elliptic umbilic, and there is one such %a
\corC{intersection} point near \corC{a} hyperbolic umbilic\corC{, in the
generic context}. 

\section{Singularities of parallel surfaces}
In this section we present our main theorem.

\subsection{Augmented distance squared functions}\label{sec:ADSU}
Let $f:(\BB{R}^n,{\bf 0})\to(\BB{R},{\bf 0})$ be a smooth function
germ. We say that a smooth function germ
$F:(\BB{R}^{n}\times\BB{R}^{r},{\bf 0})\to(\BB{R},0)$ is an \corR{{\it
unfolding}} of $f$ if $F({\bf u},{\bf 0})=f({\bf u})$. 
We define the \corR{{\it discriminant set}} of $F$ by 
\[
\mathcal{D}(F)=\biggl\{{\bf x}\in\BB{R}^r\ \corR{;}\ 
%\exists{\bf u}\in\BB{R}^n\text{ such that }
F({\bf u},{\bf x})=\pd{F}{u_1}({\bf u},{\bf x})=\cdots\pd{F}{u_n}({\bf
u},{\bf x})=0\corR{\text{ for some }{\bf u}\in U}\biggr\},
\]
where $({\bf u},{\bf
x})=(u_1,\corC{\ldots},u_n,x_1,\corC{\ldots},x_r)\in(\BB{R}^{n}\times\BB{R}^{r},{\bf
0})$. We say \colR{that $F$} is a \corR{$\mathcal{K}$-{\it versal unfolding}} if any 
unfolding $G:(\BB{R}^{n}\times\BB{R}^{s},{\bf 0})\to(\BB{R},0)$ of $f$
is representable in the form
\[
 G({\bf u},{\bf y})=h({\bf u},{\bf y})\cdot F(\Psi({\bf u},{\bf
 y}),\psi({\bf y})),
\] 
where $\Psi:(\BB{R}^n\times
\BB{R}^s,{\bf 0})\to(\BB{R}^n,{\bf 0})$ is a smooth map germ with
$\Psi({\bf u},{\bf 0})={\bf u}$, 
$\psi:(\BB{R}^s,{\bf 0})\to(\BB{R}^r,{\bf 0})$ is a smooth map germ
with $\psi({\bf 0})=0$ and $h:(\BB{R}^n \times
\BB{R}^s,{\bf 0})\to\BB{R}$ is a smooth function germ with
$h({\bf 0},{\bf 0})\ne0$ (cf. \cite[\corC{\S 8}]{A1}).
This condition is equivalent to \colR{the equality}
\[
  \mathcal{E}_n=
  \bigg\langle
  \dfrac{\partial f}{\partial u_1},\cdots,
  \dfrac{\partial f}{\partial u_n},f
  \bigg\rangle_{\mathcal{E}_n}
  +
  \bigg\langle
  \left.\dfrac{\partial F}{\partial x_1}\right|_{\BB{R}^n\times\{{\bf 0}\}},\cdots,
  \left.\dfrac{\partial F}{\partial x_r}\right|_{\BB{R}^n\times\{{\bf 0}\}}
  \bigg\rangle_{\BB{R}}
  +\mathcal{M}_n^{k+1}
\]
when $f({\bf u})$ is $k$-determined (see \cite[\S 3]{Wall} and \cite[p.75]{lu1}).
Here, $\mathcal{E}_n$ is the set of smooth function germs
$(\BB{R}^n,{\bf 0})\to\BB{R}$, which is the local ring with
the unique maximal ideal
$\mathcal{M}_n=\{f\in\mathcal{E}_n\ \corR{;}\ f({\bf 0})=0\}$.
We say that two function germs $f$ and $g:(\BB{R}^n,{\bf
0})\to(\BB{R},{\bf 0})$ are \corR{$\mathcal{K}$-{\it equivalent}} if
there exist a diffeomorphism germ $\psi:(\BB{R}^n,{\bf
0})\to(\BB{R}^n,{\bf 0})$ and a smooth function germ $h:(\BB{R}^n,{\bf
0})\to\BB{R}$ with $h({\bf 0})\ne0$ such that $g({\bf u})=h({\bf
u})\cdot f\circ \psi({\bf u})$. 
\colR{Let} $F$, $G:(\BB{R}^n\times
\BB{R}^r,{\bf 0})\to(\BB{R},{\bf 0})$ \colR{be} $\mathcal{K}$-versal
unfoldings of $\mathcal{K}$-equivalent function germs $f$, $g$,
respectively.
Then, there exist a diffeomorphism germ
$\tilde{\Psi}:(\BB{R}^n\times\BB{R}^r,{\bf 0})\to(\BB{R}^n\times\BB{R}^r,{\bf 0})$,
$({\bf u},{\bf x})\mapsto(\Psi({\bf u},{\bf x}),\psi({\bf x}))$ and a smooth 
function germ $h:(\BB{R}^n\times\BB{R}^r,{\bf 0})\to\BB{R}$ with
$h({\bf 0},{\bf 0})\ne0$ such that
\[
 G({\bf u},{\bf x})=h({\bf u},{\bf x})\cdot F(\Psi({\bf u},{\bf x}),\psi({\bf x})).
\]
(cf. \cite[\corC{\S 8}]{A1}).
Moreover, %by a simple calculation, we have
\colR{a calculation} shows \colR{the equality} $\mathcal{D}(F)=\psi(\mathcal{D}(G))$.

In order to investigate singularities of parallel surfaces, 
we consider the functions 
\begin{align*}
\Phi^{t}:U\times\BB{R}^{3}\to\BB{R},\qquad\textrm{defined by}\qquad
(u,v,x,y,z)\mapsto-\frac{1}{2}\left(\|(x,y,z)-g(u,v)\|^2-{t_{0}}^{2}\right),
\end{align*}
where $t_{0}\in\BB{R}\setminus\{0\}$,
and 
\begin{align*}
\Phi:U\times\BB{R}^{4}\to\BB{R},\qquad\textrm{defined by}\qquad
(u,v,x,y,z,t)\mapsto-\frac{1}{2}\left(\|(x,y,z)-g(u,v)\|^2-t^{2}\right).
\end{align*}
We call them \corC{{\it augmented distance squared functions}}.

%since
\corC{Calculating the discriminant set of $\Phi^t$}, we have %the
%discriminant of $\Phi^{t}$ is given by  
\[
  \mathcal{D}(\Phi^{t})=\{(x,y,z)\in\BB{R}^{3}\ \corR{;}\
 %\exists(u,v)\in\BB{R}^{2}\textrm{ such that }
 (x,y,z)=g(u,v)+t_{0}{\bf n}(u,v)\corR{\text{ for some }(u,v)\in\BB{R}^2}\},
\]
which is the parallel surface of $g$ at a distance $t_{0}$.
Besides, the discriminant set of $\Phi$ is given by
\begin{align*}
 \mathcal{D}(\Phi)=\{(x,y,z,t)\in\colB{\BB{R}^{4}} \ \corR{;}\
 %\exists(u,v)\in\BB{R}^{2}\textrm{ such that }
 (x,y,z)=g(u,v)+t{\bf n}(u,v)\corR{\text{ for some }(u,v)\in\BB{R}^2}\}.
\end{align*}
Its intersection with the \corC{hyper}plane $t=t_{0}$ is the parallel surface of $g$
at %a
distance $t_{0}$. 

We take %a point
\colB{points $p\in U$, and $q=(x_{0},y_{0},z_{0})\in\BB{R}^3$ 
or $q=(x_{0},y_{0},z_{0},t_{0})\in\BB{R}^4$} where
\[
 (x_{0},y_{0},z_{0})=% g(u_0,v_0)+t_{0}{\bf n}(u_0,v_0), 
%  \colR{\qquad t_{0}=\frac{1}{\kappa_{i}(u_0,v_0)},}
 \colB{g(p)+t_{0}{\bf n}(p), 
 \colR{\qquad t_{0}=\frac{1}{\kappa_{i}(p)},}}
\]
possibly with %$\kappa_{1}(u_0,v_0)=\kappa_{2}(u_0,v_0)$,
\colB{$\kappa_{1}(p)=\kappa_{2}(p)$,} 
and set \colB{$\varphi(u,v)=\Phi^{t}(u,v,q)$ or
$\varphi(u,v)=\Phi(u,v,q)$.} 
% \[
%  \varphi(u,v)=\Phi^{t}(u,v,q_0)\qquad\text{or}
%  \qquad\varphi(u,v)=\Phi(u,v,q_0).
% \]
Then the augmented distance functions $\Phi$ and $\Phi^t$ are the
unfoldings of $\varphi$.

\corC{If $\varphi$ is $\mathcal{K}$-equivalent to $A_2$ (\colR{resp.\ }$A_3$)
and $\Phi^t$ is a $\mathcal{K}$-versal unfolding of $\varphi$, then the
discriminant set of $\Phi^t$ is locally diffeomorphic to the
discriminant set of the versal unfolding $G:(U\times\BB{R}^{3},{\bf
0})\to(\BB{R},0)$, 
\[
 G(u,v,x,y,z)=u^3\pm v^2+x+yu
 \colR{\quad(\text{resp.\ }G(u,v,x,y,z)=u^4\pm v^2+x+yu+zu^2),}
\]
of $g(u,v)=u^3\pm v^2$ (resp.~$g(u,v)=u^4\pm v^2$). \corC{The
singularity of} the discriminant set of $G$ is \corC{the} cuspidal edge
(resp.~swallowtail).} 
% If $\varphi$ is $\mathcal{K}$-equivalent to either $A_2$ or $A_3$, and
% $\Phi^{t}:(U\times\BB{R}^{3},(u_0,v_0,q_{0}))\to(\BB{R},0)$ is a
% $\mathcal{K}$-versal unfolding of $\varphi$. Then the discriminant set
% of $\Phi^t$ is either locally diffeomorphic to the discriminant set of
% the following versal unfoldings $G:(U\times\BB{R}^{3},{\bf
% 0})\to(\BB{R},0)$ of $g(u,v)=G(u,v,{\bf 0})$:
% \begin{align*}
%  A_2&:G(u,v,x,y,z)=u^3\pm v^2+x+yu,\\
%  A_3&:G(u,v,x,y,z)=u^4\pm v^2+x+yu+zu^2.
% \end{align*}
% Their discriminant sets are a cuspidal edge and swallowtail,
% respectively.

Here, the \corR{{\it cuspidal edge}} is a set locally diffeomorphic to the image of
a map germ $CE:(\BB{R}^2,{\bf 0})\to(\BB{R}^3,{\bf
0}),\,(u,v)\mapsto(u,v^2,v^3)$ \corC{and}
the \corR{{\it swallowtail}} is a a set locally diffeomorphic to the image of a map 
germ
$SW:(\BB{R}^2,{\bf 0})\to(\BB{R}^3,{\bf
0}),\,(u,v)\mapsto(u,3v^4+uv^2,4v^3+2uv)$.
The pictures of the cuspidal edge and the swallowtail are shown in
Figure \ref{fig:sing1}. 
\begin{figure}[htbp] 
\begin{center}
 \includegraphics[width=0.6\textwidth]{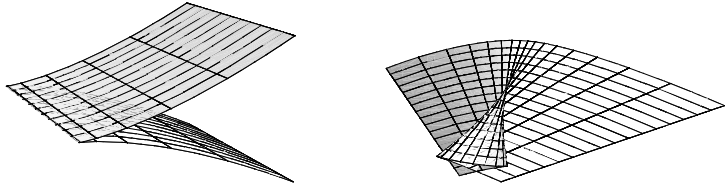}
\end{center}
\caption{From left to right: Cuspidal edge, Swallowtail.}
\label{fig:sing1}
\end{figure}

\corC{If $\varphi$ is $\mathcal{K}$-equivalent to $A_4$
(resp.~$D_4^\pm$) and
$\Phi%:(U\times\BB{R}^{4},(u_0,v_0,q_{0}))\to(\BB{R},0)
$ is a
$\mathcal{K}$-versal unfolding of $\varphi$, then the discriminant set
of $\Phi$ is locally diffeomorphic to the discriminant set of the versal
unfolding $G:(U\times\BB{R}^{4},{\bf 0})\to(\BB{R},0)$, 
\[
 G(u,v,x,y,z,t)=u^4\pm
 v^2+x+yu+zu^2+tu^3\quad
 \colR{(\text{resp.~}G(u,v,x,y,z,t)=u^2 v\pm v^3+x+y u+z v+t v^2),}
\]
of $g(u,v)=u^4\pm v^2$ (resp.~$g(u,v)=u^2v\pm v^3$).
\corC{The singularity of} the discriminant set of $G$ is a butterfly
(resp.~$D_4^\pm$ singularities).} 

%and
\corC{Here,} the \corR{{\it butterfly}} is a set locally diffeomorphic
to the image of a map germ $BF:(\BB{R}^3,{\bf 0})\to(\BB{R}^4,{\bf
0}),\,(u,v,w)\mapsto(u,5v^4+2uv+3v^2w,4v^5+uv^2+2v^3w,w)$ and %Moreover,
the \corR{{\it $4$-dimensional $D_4^\pm$ singularity}} is a set locally
diffeomorphic to the image of a map germ $FD^{\pm}:(\BB{R}^3,{\bf
0})\to(\BB{R}^4,{\bf
0}),\,(u,v,w)\mapsto(uv,u^2+2vw\pm3v^2,2u^2v+v^2w\pm2v^3,w)$. 
%respectively.

% Besides, $\varphi$ is $\mathcal{K}$-equivalent to either $A_4$ or
% $D_4^\pm$ and $\Phi:(U\times\BB{R}^{4},(u_0,v_0,q_{0}))\to(\BB{R},0)$
% is a $\mathcal{K}$-versal unfolding of $\varphi$.
% Then the discriminant set of $\Phi$ is either locally diffeomorphic to
% the discriminant set of the following versal unfoldings
% $G:(U\times\BB{R}^{4},{\bf 0})\to(\BB{R},0)$ of
% $g(u,v)=G(u,v,{\bf 0})$:
% \begin{align*}
%  A_4&:G(u,v,x,y,z,t)=u^4\pm v^2+x+yu+zu^2+tu^3,\\
%  D_4^\pm&:G(u,v,x,y,z,t)=u^2v\pm v^3+x+yu+zv+tv^2.
% \end{align*}
% Their discriminant sets are a butterfly and 4-dimensional $D_4^\pm$
% singularity, respectively.

% Here, the {\bf cuspidal edge} is a set locally diffeomorphic to the image of
% a map germ
% $CE:(\BB{R}^2,{\bf 0})\to(\BB{R}^3,{\bf 0}),\,(u,v)\mapsto(u,v^2,v^3)$,
% the {\bf swallowtail} is a a set locally diffeomorphic to the image of a map 
% germ
% $SW:(\BB{R}^2,{\bf 0})\to(\BB{R}^3,{\bf 0}),\,(u,v)\mapsto(u,3v^4+uv^2,4v^3+2uv)$,

\subsection{Criteria for singularities of fronts in $\BB{R}^3$}
It is well known that the parallel surface $g^t$ is a front.
Fronts were first studied in details by Arnol'd and Zakalyukin. 
They showed that the generic singularities of fronts in $\BB{R}^3$
are cuspidal edges and swallowtails.
Moreover, they showed that the singularities of the bifurcations in
generic one parameter families of fronts in $\BB{R}^3$ are cuspidal
lips, cuspidal beaks, cuspidal butterflies and 3-dimensional $D_4^\pm$
singularities (cf. \cite{A1}).

Here, the \corR{{\it cuspidal lips}} is a set locally diffeomorphic to
the image of a map germ $CLP:(\BB{R}^2,{\bf 0})\to(\BB{R}^3,{\bf
0}),\,(u,v)\mapsto(3u^4+2u^2v^2,u^3+uv^2,v)$, the \corR{{\it cuspidal beaks}}
is a set locally diffeomorphic to the image of a map germ
$CBK:(\BB{R}^2,{\bf 0})\to(\BB{R}^3,{\bf
0}),\,(u,v)\mapsto(3u^4-2u^2v^2,u^3-uv^2,v)$, the \corR{{\it cuspidal
butterfly}} is a set of the image of a map germ $CBF:(\BB{R}^2,{\bf
0})\to(\BB{R}^3,{\bf 0}),\,(u,v)\mapsto(4u^5+u^2v,5u^4+2uv,v)$ and the
\corR{{\it $3$-dimensional $D_4^+$ singularity}} (resp.~\corR{{\it $D_4^-$
singularity}}) is a set of the image of a map germ $TD^+:(\BB{R}^2,{\bf
0})\to(\BB{R}^3,{\bf 0}),\,(u,v)\mapsto(uv,u^2+3v^2,u^2v+v^3)$
(resp.~$TD^-:(u,v)\mapsto(uv,u^2-3v^2,u^2v-v^3)$). Their pictures are
shown in Figure \ref{fig:sing2}. 
\begin{figure}[htbp] 
\begin{center}
 \includegraphics[width=\textwidth]{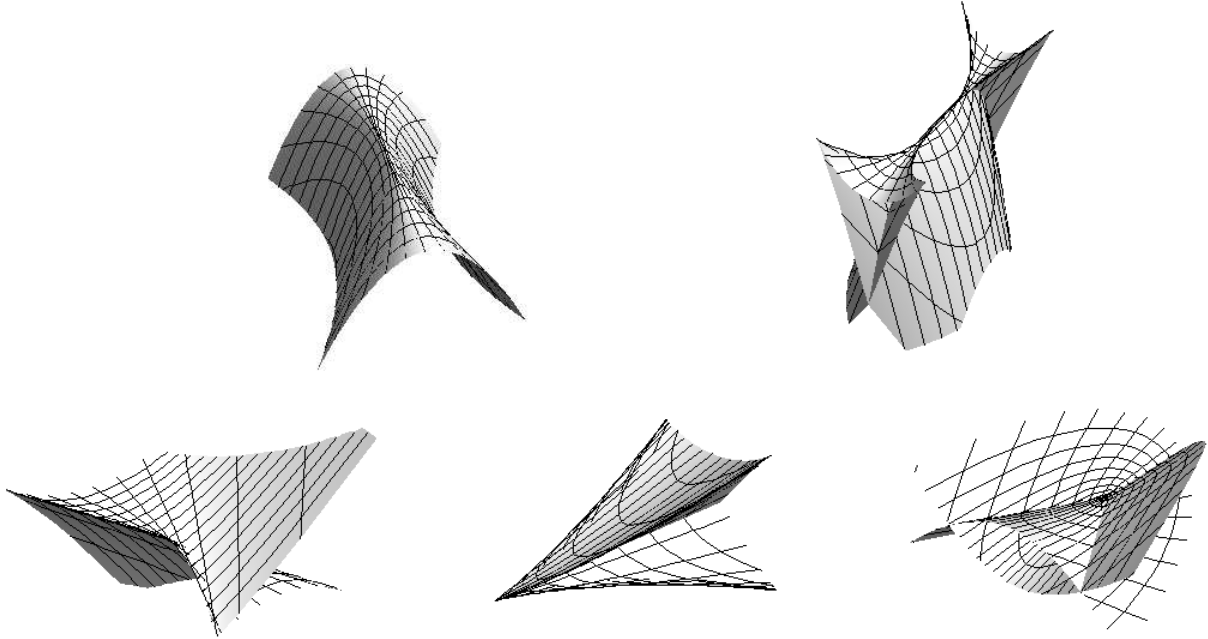}
\end{center}
\caption{From top left to bottom right: Cuspidal lips, Cuspidal beaks,
 Cuspidal butterfly, 3-dimensional $D_4^+$ singularity, 3-dimensional
 $D_4^-$ singularity.}
\label{fig:sing2}
\end{figure}

Recently, criteria for these singularities are shown in
\colR{\cite{IS, IST, KRSUY, S1}}. To present these criteria, we prepare basic
%notations
\corC{notions} of fronts in $\BB{R}^3$. A smooth map $f:U\to\BB{R}^3$ is
called a \corR{{\it front}} if there exists a unit vector field $\nu$ of
$\BB{R}^3$ along $f$ such that $L_f=(f,\nu):U\to T_1\BB{R}^3$ is a
Legendrian immersion, where $T_1\BB{R}^3$ is the unit tangent bundle of
$\BB{R}^3$ (cf. \cite{A1}, see also \cite{KRSUY}). For a front $f$, we
define \corC{a function} $\lambda:U\to\BB{R}$ by
$\lambda(u,v)=\det(f_{u},f_{v},\nu)$. \corC{The function $\lambda$ is
called a} \corR{{\it discriminant function}} \corC{of $f$.} The set of
singular points $S(f)$ of $f$ is the zero set of $\lambda$. A singular
point $p\in U$ of $f$ is \corC{said to be} %a
\corC{{\it non-degenerate}} if $d\lambda(p)\neq 0$. %holds.
Let $p$ be a non-degenerate singular point of a front $f$. Then $S(f)$
is parameterized by a %smooth
\corC{regular} curve $\gamma(t):(-\varepsilon,\varepsilon)\to U$ 
%, $\gamma(0)=p$ 
near $p$. Moreover, there exists a %smooth vector filed
\corC{a unique direction $\eta(t)\in T_{\gamma(t)}U$ up to scalar
multiplications} %along $\gamma(t)$ 
such that %$df_{\gamma(t)}(\eta(t))=0$.
\corC{$df(\eta(t))=0$.}
We call $\eta(t)$ the %{\bf null vector filed}.
\corR{{\it null direction}.} Under these notations, we present the
criterion for \corC{the} cuspidal butterfly.
\begin{theorem}
 [\cite{IS}]\label{thm:front1}
 Let $f:U\to\BB{R}^{3}$ be a front and \colB{let $p\in U$ be}
 a non-degenerate singular point of $f$. Then the germ of the front $f$
 at $p$ is $\mathcal{A}$-equivalent to the map germ $CBF$ if and only if
 $\eta\lambda(p)=\eta^2\lambda(p)=0$ and $\eta^3\lambda(p)\neq 0$.
\end{theorem}

Here, two map germs $f_1$,
$f_2:(\BB{R}^{\corC{2}},{\bf 0})\to(\BB{R}^{\corC{3}},{\bf 0})$ are
\corR{{\it $\mathcal{A}$-equivalent}} if there exist diffeomorphism
germs $\psi_1:(\BB{R}^{\corC{2}},{\bf 0})\to(\BB{R}^{\corC{2}},{\bf 0})$
and $\psi_2:(\BB{R}^{\corC{3}},{\bf 0})\to(\BB{R}^{\corC{3}},{\bf 0})$
such that %$f_1\circ\psi_2=\psi_1\circ f_2$.
\corC{$\psi_2\circ f_1=f_2\circ \psi_1$, and} %Besides,
$\eta\lambda$ denotes the directional derivative of $\lambda$ in the
direction of $\eta$.

We now turn to degenerate singularities.
Let $p$ be a degenerate singular point of the front $f$.
If $\rank(df_p)=1$, then there exists the non-zero vector field $\eta$
near $p$ such that %$df(\eta)=0$.
\corC{if $q\in S(f)$ then $df_q(\eta(q))=0$.}
Criteria for degenerate singularities are as follows:

\begin{theorem}
 [\cite{IST}]\label{thm:front2}
 Let $f:U\to\BB{R}^{3}$ be a front and \colB{let $p\in U$ be}
 a degenerate singular point of $f$. % Then the following assertions
 % hold. 
\begin{enumerate}
 \vspace{-8pt}
 \setlength{\parskip}{0cm} % 段落間
 \setlength{\itemsep}{0pt} % 項目間
 \item The germ of the front $f$ at $p$ is $\mathcal{A}$-equivalent
       to the map germ $CLP$ if and only if $\rank(df_p)=1$ and
       \colR{$\det(\mathrm{Hess}\,\lambda(p))>0$}, where 
       $\det(\mathrm{Hess}\,\lambda(p))$ denotes the determinant of the
       Hessian matrix of $\lambda$ at $p$. 
       
 \item The germ of the front $f$ at $p$ is $\mathcal{A}$-equivalent
       to the map germ $CBK$ if and only if $\rank(df_p)=1$,
       $\det(\mathrm{Hess}\,\lambda(p))<0$ and $\eta^2\lambda(p)\neq
       0$.
\end{enumerate}
\end{theorem}

\begin{theorem}
 [\cite{S1}]\label{thm:front3}
 Let $f:U\to\BB{R}^{3}$ be a front and \colB{let $p\in U$ be} a
 degenerate singular point of $f$. Then the germ of the front $f$ at 
 $p$ is $\mathcal{A}$-equivalent to the map germ $TD^+$
 $($resp.~$TD^-$$)$ if and only if $\rank(df)_p=0$ and
 $\det(\mathrm{Hess}\ \lambda(p))<0$ $($resp.~$\det(\mathrm{Hess}\
 \lambda(p))>0$$)$.
\end{theorem}

\subsection{Singularities of parallels surfaces}
Now we are ready to state our main theorem.
\begin{theorem}
 \label{thm:main1}
 Let $g:U\to\BB{R}^3$ be a regular surface and \colB{let $g^t$ be the parallel
 surface of $g$ at distance $t=1/\kappa_i(p)$}, where $U$ is an open subset of
 $\BB{R}^2$ \colB{and $p\in U$}. Assume that $\Phi$, $\Phi^t$, and $\varphi$ is defined as
 in \colR{Subsection~$\ref{sec:ADSU}$.}
 \begin{enumerate}
  \vspace{-8pt}
  \setlength{\parskip}{0cm} % 段落間
  \setlength{\itemsep}{0pt} % 項目間
  \item If \colB{$g(p)$} is neither a ridge point relative to the
        principal vector ${\bf v}_i$ nor an umbilic, %and
        %$\kappa_i(u_0,v_0)\ne0$,
        then $\varphi$ has an $A_2$ singularity at \colB{$p$.}
        %$(u_0,v_0)$.
        In this case, $\Phi^{t}$ is a $\mathcal{K}$-versal
        unfolding of $\varphi$.
        Moreover, %the parallel surface
        \colB{$g^t$} %\corC{at distance $t=1/\kappa_i(u_0,v_0)$}
        is locally diffeomorphic to the cuspidal edge at \colB{$g^t(p)$}.
  \item If \colB{$g(p)$} is a first order ridge point relative to the
        principal vector ${\bf v}_i$, %and $\kappa_i(u_0,v_0)\ne0$,
        then $\varphi$ has an $A_3$ singularity at \colB{$p$.}
        %\corC{$(u_0,v_0)$}.
        In this case, $\Phi^{t}$ is a $\mathcal{K}$-versal unfolding of
        $\varphi$ if and only if \colB{$g(p)$} is not a sub-parabolic
        point relative to the other principal vector ${\bf v}_j$.
        Moreover, %the parallel surface
        \colB{$g^t$} %\corC{at distance $t=1/\kappa_i(u_0,v_0)$}
        is locally diffeomorphic to the swallowtail at \colB{$g^t(p)$}.
  \item If \colB{$g(p)$} is a second order ridge point relative to the
        principal vector ${\bf v}_i$, %and $\kappa_i(u_0,v_0)\ne0$,
        then $\varphi$ has an $A_4$ singularity at \colB{$p$.}
        %\corC{$(u_0,v_0)$}.
        In this case, $\Phi$ is a $\mathcal{K}$-versal unfolding of
        $\varphi$ if and only if \colB{$p$} %\corC{$(u_0,v_0)$}
        is a %regular
        \colB{non-singular} point of the ridge line relative to the same
        principal \corC{vector} ${\bf v}_i$.
        Moreover, %the parallel surface
        \colB{$g^t$} %\corC{at distance $t=1/\kappa_i(u_0,v_0)$}
        is the section of the discriminant set $\mathcal{D}(\Phi)$, which
        is locally diffeomorphic to the butterfly,
        with the hyperplane \colB{$t=1/\kappa_i(p)$}.
  \item If \colB{$g(p)$} is a hyperbolic umbilic, %and not a flat umbilic
        then $\varphi$ has a $D_4^+$ singularity at \colB{$p$.}
        %\corC{$(u_0,v_0)$}.
        In this case, $\Phi$ is a $\mathcal{K}$-versal
        unfolding of $\varphi$ if and only if \colB{$g(p)$} is not a
        right-angled umbilic.
        Moreover, %the parallel surface
        \colB{$g^t$} %\corC{at distance $t=1/\kappa_i(u_0,v_0)$}
        is the section of the discriminant set $\mathcal{D}(\Phi)$, which
        is locally diffeomorphic to the $4$-dimensional $D_4^+$ singularity, 
        with the hyperplane \colB{$t=1/\kappa_i(p)$}.
  \item If \colB{$g(p)$} is an elliptic umbilic, %\corC{and} not a flat
        %umbilic,
        then $\varphi$ has a $D_4^-$ singularity at \colB{$p$.}
        %\corC{$(u_0,v_0)$}.
        In this case, $\Phi$ is a $\mathcal{K}$-versal
        unfolding of $\varphi$.
        Moreover, %the parallel surface
        $g^t$ %\corC{at distance $t=1/\kappa_i(u_0,v_0)$}
        is the section of the discriminant set $\mathcal{D}(\Phi)$, which
        is locally diffeomorphic to the $4$-dimensional $D_4^-$
        singularity, with the hyperplane \colB{$t=1/\kappa_i(p)$}.
 \end{enumerate}
\end{theorem}
A proof of this theorem is given in Section \ref{sec:proof}.

Again, we remark that the parallel surfaces $g^t$ of a regular surface
$g$ are the front.
Since the unit normal vector of %the parallel surface
$g^t$ coincides with the unit normal vector ${\bf n}$
of the initial surface $g$, the discriminant function of $g^t$ is given
by
% \begin{align}
%  \label{eq:disc_para}
\[
 \lambda(u,v)=\det(g^t_u(u,v),g^t_v(u,v),{\bf n}(u,v)).
\]
% \end{align}
Moreover, the Jacobian matrix $J_{g^t}$ of $g^{t}$ is given by
\begin{equation}
 \label{eq:jacobipara}
 J_{g^{t}}=%J_{g}-tJ_{g}\mathrm{I}^{-1}\II=J_{g}(I_{2}-t\mathrm{I}^{-1}\II),
 \colB{J_g\left(
    \left(
     \begin{array}{cc}
      1 & 0 \\
      0 & 1
     \end{array}
    \right)
    -t\left(
    \begin{array}{cc}
     E & F\\
     F & G
    \end{array}
      \right)^{-1}
    \left(
     \begin{array}{cc}
     L & M \\
     M & N
     \end{array}
    \right)
    \right),}
\end{equation}
\colB{where $J_{g}$ is the Jacobian matrix of $g$.}
%and $I_{2}$ is the $2\times 2$ identity matrix.
Applying criteria for singularities of fronts (Theorem~\ref{thm:front1}
\colB{through} \ref{thm:front3}) to %the parallel surface
$g^t$, we obtain Theorem~\ref{cor:front} as corollaries of these criteria.

\begin{theorem}
 \label{cor:front}
 Let $g:U\to\BB{R}^3$ be a regular surface and \colB{let $g^t$ be}
 the parallel surface of $g$ at distance \colB{$t=1/\kappa_i(p)$}, where $U$ is an open
 subset of $\BB{R}^2$ \colB{and $p\in U$}.
 \begin{enumerate}
   \vspace{-8pt}
   \setlength{\parskip}{0cm} % 段落間
   \setlength{\itemsep}{0pt} % 項目間
  \item Suppose that $g(p)$ is a second order ridge point
        relative to the principal vector ${\bf v}_i$ which is
        not a sub-parabolic point relative to the other principal
        direction ${\bf v}_j$. %, and that $\kappa_i(p)\ne0$.
        Then %the parallel surface
        $g^t$ %at distance $t=1/\kappa_i(p)$
        is locally diffeomorphic to the cuspidal butterfly at
        $g^t(p)$. 
  \item Suppose that $g(p)$ is a ridge point relative to the
        principal direction ${\bf v}_i$ and %a
        sub-parabolic point relative to the other principal direction
        ${\bf v}_j$. %, \corC{and that $\kappa_i(p)\ne0$}.
        Then %the parallel surface
        $g^t$ %at distance $t=1/\kappa_1(p)$
        is locally diffeomorphic to the cuspidal lips
        $($resp.~cuspidal  beaks$)$ at $g^t(p)$ if
        $\det(\mathrm{Hess}_{({\bf v}_1,{\bf
        v}_2)}\kappa_i(p))>0$ $($resp.~$\det(\mathrm{Hess}_{({\bf
        v}_1,{\bf v}_2)}\kappa_i(p))<0$ and the order of ridge is
        one$)$, where $\mathrm{Hess}_{({\bf v}_1,{\bf v}_2)}\kappa_i$
        is the Hessian matrix of $\kappa_i$ with respect to ${\bf v}_1$
        and ${\bf v}_2$. 
  \item Suppose that $g(p)$ is an umbilic. %which is not a flat umbilic.
        Then %the parallel surface
        $g^t$ %at distance $t=1/\kappa_1(0,0)=1/\kappa_2(0,0)$
        is locally diffeomorphic
        to a 3-dimensional $D_4^+$ singularity $($resp.~$D_4^-$
        singularity$)$ at $g^t(p)$ if $g(p)$ is a hyperbolic umbilic
        $($resp.~elliptic umbilic$)$.  
 \end{enumerate}
\end{theorem}
\begin{proof}
 (1)\quad We may assume that $p=(0,0)$ and that the initial regular
 surface $g$ given in Monge form as in \eqref{eq:monge}.
 We remark that $k_1\ne k_2$.
 Now we prove \colR{the theorem} in the case $t=1/\kappa_1(0,0)=1/k_1$.
 From Lemma~\ref{lem:ridge} and \ref{lem:subpara},
 we have
 \begin{align}
  \label{eq:2-ridge}
  \begin{split}
   a_{30}=3{a_{21}}^2+(a_{40}-3{k_1}^3)(k_1-k_2)&=0,\\
   15{a_{21}^2}a_{12}+10a_{21}a_{31}(k_1-k_2)+a_{50}(k_1-k_2)^2&\ne0,\quad\text{and}\quad
   a_{21}\ne0.
  \end{split}
 \end{align}
 Suppose that $t=1/k_1$.
 Then we have $\lambda(0,0)=0$.
 Moreover, from \eqref{eq:2-ridge}, we have \colB{$\lambda_u(0,0)=0$ and
 $\lambda_v(0,0)\ne0$.} 
%   \begin{align}
%    \label{eq:partialdiscpara}
%    \lambda_u(0,0)=\dfrac{a_{30}(k_2-k_1)}{{k_1}^2}\corC{=0}\quad\corC{\text{and}}\quad
%    \lambda_v(0,0)=\dfrac{a_{21}(k_2-k_1)}{{k_1}^2}\corC{\ne0}.
%   \end{align}
 It turns out that $(0,0)$ is a non-degenerate singular point of $g^t$.
 Therefore, the set of singular points of $g^t$ is a locally smooth
 curve near $(0,0)$, which is the CPC line $\Sigma_{k_1}$, and there
 exists a null direction $\eta$ with $dg^t(\eta)=0$ along this smooth
 curve.
 It follows form  \eqref{eq:jacobipara} that the null direction $\eta$
 has the same direction as the principal vector ${\bf v}_1$.
 From \eqref{eq:2-ridge}, we have \colB{${\bf v}_1\lambda(0,0)={{\bf
 v}_1}^2\lambda(0,0)=0$ and ${{\bf v}_1}^3\lambda(0,0)\ne0$.}
%   \begin{align*}
%    \begin{split}
%     {\bf v}_1\lambda(0,0)&=-\dfrac{a_{30}(k_1-k_2)^2}{{k_1}^2}\corC{=0},\\
%     {{\bf
%     v}_1}^2\lambda(0,0)
%     &=-\dfrac{(k_1-k_2)^2[a_{30}(a_{30}-3a_{12})+3{a_{21}}^2+(a_{40}-3{k_1}^3)(k_1-k_2)]}
%     {{k_1}^2}\corC{=0},\quad\corC{\text{and}}\\
%      {{\bf v}_1}^3\lambda(0,0)
%     &=-\dfrac{(k_1-k_2)^2[15{a_{21}^2}a_{12}+10a_{21}a_{31}(k_1-k_2)+a_{50}(k_1-k_2)^2]}{{k_1}^2}\ne
%     0.
%    \end{split}
%   \end{align*}
 Therefore, we obtain %that
 $\eta\lambda(0,0)=\eta^2\lambda(0,0)=0$, $\eta^3\lambda(0,0)\ne0$.
 If the two map germs are $\mathcal{A}$-equivalent, their images are
 locally diffeomorphic.
 %Hence,
 \colB{By Theorem~\ref{thm:front1}, %the parallel surface
 $g^t$} %at distance $t=1/k_1$
 is locally diffeomorphic to the cuspidal butterfly \colB{at $g^t(p)$}.

 (2)\quad We may assume that $p=(0,0)$ and that the initial regular
 surface $g$ given in Monge form as in \eqref{eq:monge}.
 We remark that $k_1\ne k_2$.
 Now we prove \colR{the theorem} in the case $t=1/\kappa_1(0,0)=1/k_1$.
 From Lemma~\ref{lem:ridge} and \ref{lem:subpara}, we have 
 \begin{equation}
  \label{eq:ridge_sub-para}
  a_{30}=a_{21}=0.
 \end{equation}
 Suppose that $t=1/k_1$.
 Then we have $\lambda(0,0)=0$ and
 \[
   J_{g^t}(0,0)=
   \begin{pmatrix}
    0 & 0 \\
    0 & (k_1-k_2)/{k_1}\\
    0 & 0\\
   \end{pmatrix}.
 \]
 Moreover, from \eqref{eq:ridge_sub-para}, we have
 $\lambda_u(0,0)=\lambda_v(0,0)=0$. 
 It follows that $(0,0)$ is a degenerate singular point of $g^t$ with
 $\rank(dg^t_p)=1$. 
 Using \eqref{eq:ridge_sub-para}, we obtain %that
%  \[
%  \det(\text{Hess}\lambda(0,0))
%  =\frac{(k_1-k_2)^2}{{k_1}^4}
%  \begin{array}{|cc|}
%   a_{40}-3{k_1}^3 & a_{31} \\
%   a_{31} & \dfrac{2{a_{12}}^2+(a_{22}-k_1{k_2}^2)}{k_1-k_2}
%  \end{array}
% \]
% and
\colB{
 \begin{align}
 \label{eq:hesskappa}
  \det(\text{Hess}_{({\bf v}_1,{\bf v}_2)}\kappa_1(0,0))
  &=
  \begin{array}{|cc|}
   a_{40}-3{k_1}^3 & a_{31} \\
   a_{31} & \dfrac{2{a_{12}}^2+(a_{22}-k_1{k_2}^2)}{k_1-k_2}
  \end{array}\\\nonumber
  &=\frac{{k_1}^4}{(k_1-k_2)^2}\det(\text{Hess}\lambda(0,0)).
 \end{align}}
 Therefore, the sign of $\det(\text{Hess}\lambda(0,0))$ 
 is the same as \colR{that} 
 of $\det(\text{Hess}_{({\bf v}_1,{\bf v}_2)}\kappa_1(0,0))$.
 Besides, since $\rank(dg^t_p)=1$, there exists a
 non-zero vector $\eta$ with $dg^t_p(\eta)=0$.
 From \eqref{eq:jacobipara}, the non-zero vector $\eta$ has the same
 direction as the principal vector ${\bf v}_1$. 
% Hence, $\eta^2\lambda(0,0)\ne0$ if and only if ${{\bf
% v}_1}^2\lambda(0,0)\ne0$. 
% From \eqref{eq:ridge_sub-para}, we have 
%  \[
%  {{\bf v}_1}^2\lambda(0,0)=-\dfrac{(a_{40}-3{k_1}^3)(k_1-k_2)^3}{{k_1}^2}.  
%  \]
 \colB{Using \eqref{eq:ridge_sub-para}, we conclude that $(0,0)$ is a
 first order blue ridge point relative to ${\bf v}_1$ if and only if
 ${{\bf v}_1}^2\lambda(0,0)\ne0$, that is, ${\eta}^2\lambda(0,0)\ne0$.} 
%  Therefore, this shows that $(0,0)$ is a first order blue ridge point if
%  and only if $\eta^2\lambda(0,0)\ne0$ (cf.~Lemma~\ref{lem:ridge}). 
 Applying Theorem \ref{thm:front2} to the argument indicated above, we
 obtain %the assertion of this theorem.
 (2).

 (3)\quad  We may assume that $p=(0,0)$ and that the initial regular
 surface $g$ given in Monge form as in 
 \eqref{eq:umbmonge}.
 We remark that $\kappa_1(0,0)=\kappa_2(0,0)=k$.
 Suppose that $t=1/k$.
 Then we have \colB{$\lambda(0,0)=\lambda_u(0,0)=\lambda_v(0,0)=0$ and
 $\rank(J_{g^t}(0,0))=0$.}  
%  \[
%  \colR{\lambda(0,0)=0,}\quad 
%  \lambda_u(0,0)=t(kt-1)(a_{30}+a_{21})\corC{=0},\quad
%  \lambda_v(0,0)=t(kt-1)(a_{12}+a_{03})\corC{=0},
%  \]
%  \colR{and}
%  \[
%  J_{g^t}(0,0)=
%  \begin{pmatrix}
%   1-kt & 0 \\
%   0 & 1-kt \\
%   0 & 0\\
%  \end{pmatrix}\corC{=
%  \begin{pmatrix}
%   0 & 0 \\
%   0 & 0 \\
%   0 & 0\\
%  \end{pmatrix}}.
%  \]
 Hence, $(0,0)$ is a degenerate singular point
 of $g^t$ with $\rank(dg^t_p)=0$. %By a computation, we have
 Moreover, we have
%  \[
%   \det(\mathrm{Hess}\lambda(0,0))=%-\dfrac1{{k_1}^4}
%  -\dfrac1{\corC{k^4}}
%  ({a_{30}}^2{a_{03}}^2-6a_{30}a_{21}a_{12}a_{03}+4a_{30}{a_{12}}^3+4{a_{21}}^3a_{03}-3{a_{21}}^2{a_{12}}^2)
%   \corC{=-\dfrac1{k^4}\Gamma,}
%  \]
 \colB{$\det(\mathrm{Hess}\lambda(0,0))=-\Gamma/{k_1}^4$}, where
 $\Gamma$ is as in %Section
 \colB{Subsection}~\ref{sec:umbilics}.
 It follows
 that $\det(\mathrm{Hess}\lambda(0,0))<0$
 (resp.~$\det(\mathrm{Hess}\lambda(0,0))>0$) if and only if $g(0,0)$ is
 a hyperbolic \colB{(resp.~elliptic)} umbilic %(resp.~elliptic umblic).
 Therefore, using Theorem~\ref{thm:front3}, we obtain (3).
\end{proof}

\begin{remark}
 %Let $g(p)$ be
 \corC{Suppose that $g(p)$ is} a ridge point relative to the principal direction
 ${\bf v}_i$ and %a
 sub-parabolic point relative to the other principal
 direction ${\bf v}_j$.
 It follow from \eqref{eq:rline}, \eqref{eq:spline} and
 \eqref{eq:hesskappa} that
 $\det(\mathrm{Hess}_{({\bf v}_1,{\bf v}_2)}\kappa_i(p))=0$ if and only if
 the ridge line relative to ${\bf v}_i$ and the sub-parabolic line
 relative to ${\bf v}_j$ are tangent at $p$.
\end{remark}

These theorems imply that the configuration of CPC lines, ridge lines,
and sub-parabolic lines determines types of singularities of parallel
surfaces.
For example, it follows from Theorem~\ref{thm:main1}~(1) and
\colB{Lemma~\ref{lem:CPCline}} %~(1)
that if the CPC line \colB{$\Sigma_{\kappa_i(p)}$} %($\kappa_i(u_0,v_0)\ne0$)
does not meet the ridge line relative ${\bf v}_i$ at \colB{$p$} then the
parallel surface $g^t$ at distance \colB{$t=1/\kappa_i(p)$} is the
cuspidal edge at \colB{$g^t(p)$}. 
Moreover, it follows from Theorem~\ref{thm:main1}~(2) and
\colB{Lemma~\ref{lem:CPCline}} %~(1)
that if CPC line \colB{$\Sigma_{\kappa_i(p)}$} %($\kappa_i(u_0,v_0)\ne0$)
crosses the ridge line relative to the principal vector ${\bf v}_i$ and
does not cross the sub-parabolic line relative to the other principal
vector ${\bf v}_j$ at \colB{$p$} %\colR{$(u_0,v_0)$,}
then the parallel surface $g^t$ at distance \colB{$t=1/\kappa_i(p)$} is
the swallowtail at \colB{$g^t(p)$}. 
Therefore, Figure~\ref{fig:CPCLineAtUmbilic}~(i) and (ii) show
that there are three swallowtails near \colB{$g^t(p)$} on the parallel
surface $g^t$ at distance \colB{$t=1/(\kappa_i(p)\pm\varepsilon)$} if
\colB{$g(p)$} is an elliptic umbilic. %which is not flat.
Similarly, Figure~\ref{fig:CPCLineAtUmbilic}~(iii) \colR{through} (v)
show that there is one swallowtail near \colB{$g^t(p)$} on the parallel
surface $g^t$ at distance \colB{$t=1/(\kappa_i(p)\pm\varepsilon)$} if
\colB{$g(p)$} is a hyperbolic umbilic which is %either flat nor
\colB{not} right-angled. 
These bifurcations of parallel surfaces near umbilics are depicted in
Figure~\ref{fig:bifurcationD4}.
These are also shown in \cite[\corR{p.~384}]%[page 384]
{A1}.
\begin{figure}[htbp] 
 \begin{center}
  \includegraphics[width=\textwidth]{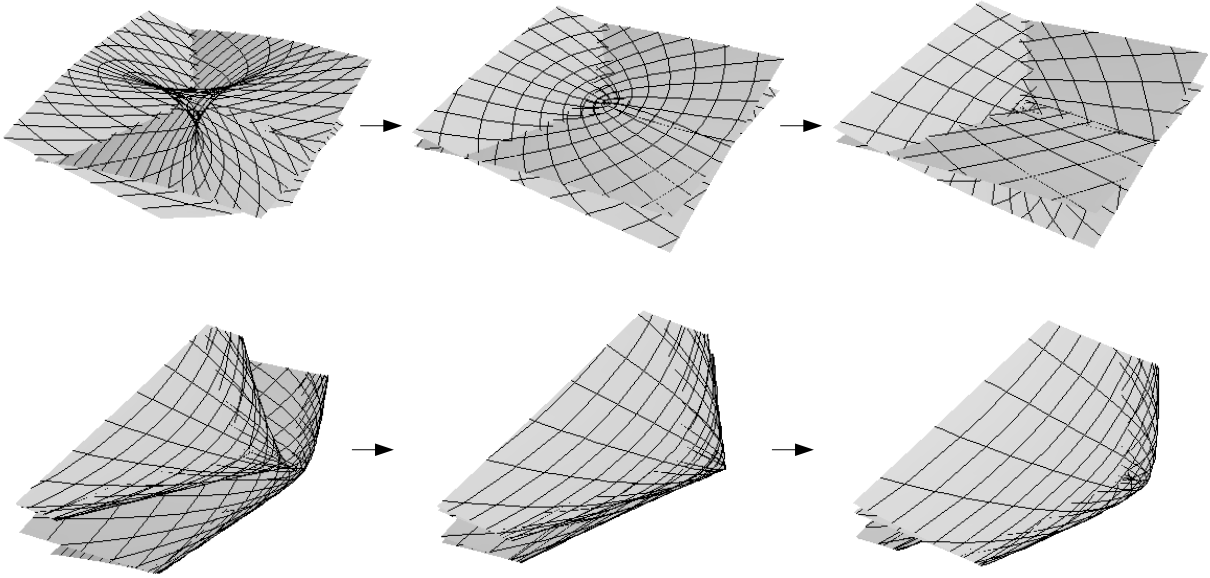}
 \end{center}
 \caption{From top to bottom: Elliptic umbilic, Hyperbolic umbilic.}
 \label{fig:bifurcationD4}
\end{figure}
 
\section{Criteria for $A_1$, $A_2$, $A_3$, $A_4$ and  $D_4^\pm$ singularities}
Before we present proof of Theorem~\ref{thm:main1}, we shall provide a convenient criteria
for $A_{\leq 4}$ and $D_{4}$ singularities in this section.
%In this Section we shall provide a convenient criteria for $A_k$ ($k\leq
%4$) and $D_4$ singularity.

We consider the function $f:(\BB{R}^2,0)\to(\BB{R},0)$ 
whose Taylor expansion at $(0,0)$ is 
\[
f(u,v)=\sum_{i,j}\frac1{i!j!}c_{ij}u^iv^j.
\]

\subsection{Criteria for $A_k$-singularities ($k\leq 4$)}
We assume that $f$ is singular at $(0,0)$ (i.e., $c_{10}=c_{01}=0$).
It is \corR{well known} that the function $f$ has an $A_1$-singularity at
$(0,0)$ if and only if
$$
\begin{pmatrix}
 c_{20}&c_{11}\\
 c_{11}&c_{02}
\end{pmatrix}
$$
is of full rank.
Now we set 
\begin{align*}
 c_{n}(u,v):=\sum_{i+j=n}\frac1{i!j!}c_{ij}u^iv^j.
\end{align*}
It is easy to see that the following conditions are equivalent.
\begin{enumerate}
 \vspace{-8pt}
 \setlength{\parskip}{0cm} % 段落間
 \setlength{\itemsep}{0pt} % 項目間
 \item  The matrix
        $\begin{pmatrix}
          c_{20}&c_{11}\\
          c_{11}&c_{02}
         \end{pmatrix}$
        is of rank 1.
        
 \item  There %is
        \corC{exists} a non-zero vector $(\lambda,\mu)$ \corC{such that}
        %with
        $\begin{pmatrix}
          c_{20} & c_{11}\\
          c_{11} & c_{02}
         \end{pmatrix}
        \begin{pmatrix}
         \lambda \\ \mu
        \end{pmatrix}=
        \begin{pmatrix}
         0 \\ 0
        \end{pmatrix}.$
        
 \item There %are
        \corC{exist} a non-zero vector $(\lambda,\mu)$ and non-zero real
        number $s$ \corC{such that} %with
        \begin{align}
         \label{eq:MatrixA2}
         \begin{pmatrix}
          c_{20} & c_{11}\\
          c_{11} & c_{02}
         \end{pmatrix}=
         s\begin{pmatrix}
           \mu^2 & -\lambda \mu \\
           -\lambda \mu & \lambda^2
          \end{pmatrix}.
        \end{align}
\end{enumerate}
The rank of the Hesse's matrix of $f$ is 1 if and only if one of these
conditions holds.
Under this assumption, we have the followings.
\begin{theorem}
 \label{thm:CriteriaAk}
 \begin{enumerate}
  \vspace{-8pt}
  \setlength{\parskip}{0cm} % 段落間
  \setlength{\itemsep}{0pt} % 項目間
 \item  The function $f$ is $A_2$-singularity at $(0,0)$ if and only
        if $c_3(\lambda,\mu)\ne0$.
        
  \item  The function $f$ is $A_3$-singularity at $(0,0)$ if and
        only if $c_3(\lambda,\mu)=0$, 
        \[
        \hat{c}_4(\lambda,\mu)
        :=c_4(\lambda,\mu)
        +\frac1{8s}
        \begin{vmatrix}
         \mu^2 & -\lambda\mu & \lambda^2\\
         c_{30} & c_{21} & c_{12}\\
         c_{21} & c_{12} & c_{03}
        \end{vmatrix}\ne0.
        \]
        
  \item  The function $f$ is $A_4$-singularity at $(0,0)$ if and only if 
        $c_3(\lambda,\mu)=\hat{c}_4(\lambda,\mu)=0$ and one of the following
        conditions holds.
        \begin{enumerate}
         \vspace{-8pt}
         \item  $\lambda\ne0$,
                $c_5(\lambda,\mu)
                -\dfrac1{s\lambda^2}c_{4v}(\lambda,\mu)c_{3v}(\lambda,\mu)
                +\dfrac1{2s^2\lambda^4}c_{3v}(\lambda,\mu)^2c_{3vv}(\lambda,\mu)$,
         \item $\mu\ne0$,
                $c_5(\lambda,\mu)
                -\dfrac1{s\mu^2}c_{4u}(\lambda,\mu)c_{3u}(\lambda,\mu)
                +\dfrac1{2s^2\mu^4}c_{3u}(\lambda,\mu)^2c_{3uu}(\lambda,\mu)$.
        \end{enumerate}
 \end{enumerate}
 
 Here, $(\lambda,\mu)$ is a non-zero vector and $s$ is a non-zero real
 number that satisfy \eqref{eq:MatrixA2}.
\end{theorem}
\begin{proof}
 (1)\quad If $\lambda\ne0$, the coefficient of $u^2$, $v^2$, and $u^3$ in
 $f(u,v+(\mu/\lambda)u)$ are 0, %$\frac12{s\lambda^2}$,
 \corR{$s\lambda^2/2$}, and %$\frac1{\lambda^3}c_3(\lambda,\mu)$
 \colR{$c_3(\lambda,\mu)/\lambda^3$,} respectively.
 Hence, we obtain the result.
 The case that $\mu\ne0$ is similar.
 
 (2)\quad We assume that $c_3(\lambda,\mu)=0$.
 Suppose that $\lambda\ne0$. 
 Setting $c=%\frac1{s\lambda^4}
 \corR{c_{3v}(\lambda,\mu)/(s\lambda^4)}$,
 we obtain that 
 the coefficients of $v^2$, $u^2v$, and $u^4$ in 
 $f(u,v+(\mu/\lambda)u-cu^2)$
 are %$\frac12{s\lambda^2}$,
 \corR{$s\lambda^2/2$}, $0$, and 
 \begin{align}
  \label{u4}
  \frac1{\lambda^4}\left(c_4(\lambda,\mu)-\frac1{2s\lambda^2}c_{3v}(\lambda,\mu)^2\right),
 \end{align}
 respectively. 
 Since 
 $$
 \lambda^2\begin{vmatrix}
       \lambda^2&-\lambda\mu&\mu^2\\
       c_{30}&c_{21}&c_{12}\\
       c_{21}&c_{12}&c_{03}
      \end{vmatrix}
 +4c_{3v}(\lambda,\mu)^2
 =6c_{3vv}(\lambda,\mu)c_3(\lambda,\mu),
 $$
 $\hat{c}_4(\lambda,\mu)\ne0$ implies that \eqref{u4} is not zero.
 The case that $\mu\ne0$ is similar.
 
 (3)\quad
 We keep the notation above and assume 
 $c_3(\lambda,\mu)=\hat{c}_4(\lambda,\mu)=0$. 
 We shall consider case (a).
 (Case (b) is similar and we omit the detail.)
 If $\lambda\ne0$, the coefficients of $v^2$, $u^2v$, $u^4$, and $u^5$ in 
 $f(u,v+(\mu/\lambda)u-cu^2)$
 are %$\frac12s{\lambda}^2$,
 \corR{$s\lambda^2/2$}, $0$, $0$, and
  \begin{align*}
  \frac1{\lambda^5}\left(c_5(\lambda,\mu)
  -\frac1{s\lambda^2}c_{4v}(\lambda,\mu)c_{3v}(\lambda,\mu)
  +\frac1{2s^2\lambda^4}c_{3v}(\lambda,\mu)^2c_{3vv}(\lambda,\mu)\right),
 \end{align*}
 respectively. \corC{The case that $\mu\ne$ is similar.}\\
 \end{proof}

\subsection{Criterion for $D_4^\pm$-singularity}
We assume that $c_{10}=c_{01}=c_{20}=c_{11}=c_{02}=0$.
Then $f$ is at least $D_4$-singularity at $(0,0)$.
We have the following.
\begin{theorem}
\label{thm:CriteriaD4}
 The function $f$ is $D_4^+$-singularity \corR{$($}resp.~$D_4^-$-singularity\corR{$)$} at
 $(0,0)$ %if the determinant of the $4\times 4$ matrix
 \corC{if and only if}
\begin{align}
 \label{eq:disc_D4}
  \begin{vmatrix}
  c_{30}&2c_{21}&c_{12}&0\\
  0&c_{30}&2c_{21}&c_{12}\\
  c_{21}&2c_{12}&c_{03}&0\\
  0&c_{21}&2c_{12}&c_{03}
 \end{vmatrix}
\end{align}
 takes positive values \corR{$($}resp.~negative values\corR{$)$}.
\end{theorem}
\begin{proof}
 The function $f$ is $D_4^+$-singularity \corC{or $D_4^-$-singularity}
 at $(0,0)$ if the cubic part $c_3$ of $f$ has one real root %, that is
 %$\Delta<0$, or $D_4^-$-singularity at $(0,0)$ if $c_3$ has
 \corC{or} three real roots, %that is $\Delta>0$.
 \corC{respectively.}
 The discriminant $\Delta$ of $c_3$ is given by
 \[
  \Delta=-\frac1{48}({a_{30}}^2{a_{03}}^2-6a_{03}a_{21}a_{12}a_{30}+4a_{30}{a_{12}}^3+4{a_{21}}^3a_{03}-3{a_{21}}^2{a_{12}}^2).
 \]
\corC{Expanding \eqref{eq:disc_D4}, we have
 \[
  \begin{vmatrix}
  c_{30}&2c_{21}&c_{12}&0\\
  0&c_{30}&2c_{21}&c_{12}\\
  c_{21}&2c_{12}&c_{03}&0\\
  0&c_{21}&2c_{12}&c_{03}
 \end{vmatrix}=-48\Delta,
 \] and we complete the proof.}
%  The function $f$ is $D_4^+$-singularity at $(0,0)$ if the cubic part
%  $c_3$ of $f$ has one real root, that is $\Delta<0$, or $D_4^-$-singularity at
%  $(0,0)$ if $c_3$ has three real roots, that is $\Delta>0$.
%  Besides, $c_3$ has a double root if and only if $c_{3u}$ and
%  $c_{3v}$ has a common root.
%  Therefore, $c_3$ has a double root if and only if the determinant
%  of the following Sylvester matrix $S_{c_{3u},c_{3v}}$ associated to
%  $c_{3u}$ and $c_{3v}$ equals zero:
%  \[
%   S_{c_{3u},c_{3v}}=\frac12
%  \begin{pmatrix}
%   a_{30} & 2a_{21} & a_{12} & 0\\
%   0 & a_{30} & 2a_{21} & a_{12}\\
%   a_{21} & 2a_{12} & a_{03} & 0\\
%   0 & a_{21} & 2a_{12} & a_{03}\\
%  \end{pmatrix}.
%  \]
%  By a simple calculation, we have
%  \[
%   \det(S_{c_{3u},c_{3v}})=-3\Delta.
%  \]
 \end{proof}

\section{Singularities of $\varphi$ and $\mathcal K$-versality}
\label{sec:proof}
\colB{In this section we give the proof of Theorem~\ref{thm:main1}.}
Let $g$ be given in %a 
Monge from as (\ref{eq:monge}).
If we write down $\Phi$ as 
\[
  \Phi=c_{00}+xu+yv+\frac{1}{2}(\hat{k}_{1}u^{2}+\hat{k}_{2}v^{2})+\sum_{i+j\ge3}
  \frac{1}{i!j!}c_{ij}u^{i}v^{j}, 
\]then we obtain that 
\begin{align*}
 c_{00}&=\frac{t^{2}-x^{2}-y^{2}-z^{2}}{2},\quad
 \hat{k}_{i}=k_{i}z-1~~(i=1,2),\quad c_{ij}=a_{ij}z~~(i+j=3),\\
 c_{40}&=a_{40}z-3{k_1}^2,\quad c_{31}=a_{31}z,\quad c_{22}=a_{22}z-k_1k_2,\quad c_{13}=a_{13}z,\\
 c_{04}&=a_{04}z-3{k_2}^2,\quad
 c_{50}=a_{50}z-10k_1a_{30},\quad
 c_{05}=a_{05}z-10k_2a_{03}.
\end{align*}
We recall that
we take %a point
\colB{points $p\in U$, and $q=(x_{0},y_{0},z_{0})\in\BB{R}^3$ 
or $q=(x_{0},y_{0},z_{0},t_{0})\in\BB{R}^4$} where
\[
 (x_{0},y_{0},z_{0})=% g(u_0,v_0)+t_{0}{\bf n}(u_0,v_0), 
%  \colR{\qquad t_{0}=\frac{1}{\kappa_{i}(u_0,v_0)},}
 \colB{g(p)+t_{0}{\bf n}(p), 
 \colR{\qquad t_{0}=\frac{1}{\kappa_{i}(p)},}}
\]
% we take a point $q_{0}=(x_{0},y_{0},z_{0})$ 
% or $q_{0}=(x_{0},y_{0},z_{0},t_{0})$, where
% \[
%  (x_{0},y_{0},z_{0})=g(u_0,v_0)+t_{0}{\bf n}(u_0,v_0)\quad{\corC{\text{and}}}\quad
%  t_{0}=\frac{1}{\kappa_{i}(u_0,v_0)},
% \]
and that we set $\varphi(u,v)=\Phi(u,v,\colB{q})$ or
$\varphi(u,v)=\Phi^{t}(u,v,\colB{q})$. 
Now we assume that %$(u_0,v_0)
\colB{$p=(0,0)$}.
So we have 
% \[
%  (x_0,y_0,z_0)=\left(0,0,\frac1k_i\right),\quad t_0=\frac1k_i.
% \]
\colB{$(x_0,y_0,z_0)=(0,0,1/k_i)$ and $t_0=1/k_i$.}
We note that $\Phi$ (resp.~$\Phi^t$) is a $\mathcal{K}$-versal unfolding
of $\varphi$ if and only if
\begin{align*}
 &\mathcal{E}_{2}=
 \langle\varphi,\varphi_{u},\varphi_{v}\rangle_{\mathcal{E}_{2}}
 +\langle
 \Phi_{x}|_{\BB{R}^{2}\times \colB{q}},
 \Phi_{y}|_{\BB{R}^{2}\times \colB{q}},
 \Phi_{z}|_{\BB{R}^{2}\times \colB{q}},
 \Phi_{t}|_{\BB{R}^{2}\times \colB{q}}
 \rangle_{\BB{R}}
 +\langle u,v\rangle^{k+1} \\
%\label{VerPhi^t}
 &(\corC{\text{resp.~}}\mathcal{E}_{2}=
 \langle\varphi,\varphi_{u},\varphi_{v}\rangle_{\mathcal{E}_{2}}
 +\langle
 \Phi^{t}_{x}|_{\BB{R}^{2}\times \colB{q}},
 \Phi^{t}_{y}|_{\BB{R}^{2}\times \colB{q}},
 \Phi^{t}_{z}|_{\BB{R}^{2}\times \colB{q}}
 \rangle_{\BB{R}}
 +\langle u,v\rangle^{k+1}) 
\end{align*}
%respectively, 
when $\varphi$ is $k$-determined.
\colB{To show $\mathcal{K}$-versality of $\Phi$ and $\Phi^t$, it is
enough to check these conditioins.
We skip the proofs of (1) and (2), since the proofs are similar to that
of (3). 
The proof of (5) is also omitted, since it is completely parallel to
that of (4).} 

\begin{proof}
 [Proof of Theorem \corR{$\ref{thm:main1}$ $(3)$}]
 From Theorem \ref{thm:CriteriaAk} (3), $\varphi$ is
 $\mathcal{K}$-equivalent to $A_4$ at $(0,0)$ if and only
 if one of the following conditions holds\corC{:}
 \begin{itemize}
 \item[(a)] $\hat{k}_1=0$, $\hat{k}_2\ne0$, 
            $c_{30}=0$, 
            $\hat{k}_2c_{40}-3{c_{21}}^2=0$, 
            ${\hat{k}_2}^{2}c_{50}-10\hat{k}_2c_{21}c_{31}+15{c_{21}}^{2}c_{12}\ne0$\corC{;}
 \item[(b)] $\hat{k}_1\ne0$, $\hat{k}_2=0$, 
            $c_{03}=0$, $\hat{k}_1c_{04}-3{c_{12}}^2=0$,
            ${\hat{k}_1}^{2}c_{05}-10\hat{k}_1c_{12}c_{13}+15c_{21}{c_{12}}^{2}\ne0$.
 \end{itemize}
 We work on Case (a).
 (Case (b) is similar and we omit the detail.) 
 \colB{Condition (a) is} %are
 equivalent to 
\begin{align*}
   &z_0=1/k_1,\quad k_1\ne k_2,\quad a_{30}=0,\quad
  3{a_{21}^{2}}+(a_{40}-3{k_1}^{3})(k_1-k_2)=0,\\
  &15{a_{21}}^{2}a_{12}+10a_{21}a_{31}(k_1-k_2)^{2}+a_{50}(k_1-k_2)^{2}\ne 0,
\end{align*}
 in the original coefficients of the Monge form.
 By Lemma \ref{lem:ridge}, we obtain the first assertion.

 \colB{Let us prove $\mathcal{K}$-versality of $\Phi$.}
 \colB{We assume that $\varphi$ has an $A_4$-singularity at $(0,0)$.}
 We next remark that $A_4$-singularity is 5-determined. 
 To show $\mathcal K$-versality of $\Phi$, it is enough to verify that
 \begin{equation}
  \label{eq:VerA4}
  \mathcal E_2=
  \langle\varphi_u,\varphi_v,\varphi\rangle_{\mathcal E_2}
  +
  \langle
  {\Phi}_x |_{\BB{R}^2\times \colB{q}},
  {\Phi}_y|_{\BB{R}^2\times \colB{q}},
  {\Phi}_z|_{\BB{R}^2\times \colB{q}},
  {\Phi}_t|_{\BB{R}^2\times \colB{q}}
  \rangle_{\BB{R}}
  \colR{+\langle u,v\rangle^6.}
 \end{equation}
 Setting $c=c_{21}/(2\hat{k}_2)$ and replacing $v$ by $v-cu^{2}$, 
 we see that the coefficients of $u^{i}v^{j}$ of functions appearing in
 (\ref{eq:VerA4}) are given by the following table:
\begin{center}
 \hspace*{-0pt}\begin{tabular}{c|c|cc|ccc|cccc|c|c}
   \hline
   &1&$u$&$v$&$u^2$&$uv$&$v^2$&$u^3$&$u^2v$&$uv^2$&$v^3$&$u^4$&$u^5$\\
   \hline
   $\Phi_x$&0&\fbox{1}&0&0&0&0&0&0&0&0&0&0\\
   $\Phi_y$&0&0&\fbox{1}&$-c$&0&0&0&0&0&0&0&0\\
   $\Phi_z$&$-z_0$&0&0&\fbox{$\frac12k_1$}&0&$\frac12k_2$&
   0&$\ast$&$\ast$&$\ast$&$\ast$&$\ast$\\
   $\Phi_t$&\fbox{$t_0$}&0&0&0&0&0&0&0&0&0&0&0\\
   \hline
   $\varphi_u$&0&0&0&0&0&$\frac12c_{12}$&
   0&$\frac12\hat{c}_{31}$&$\frac12\hat{c}_{22}$&$\frac16c_{13}$&
   \fbox{$\frac1{24}\hat{c}_{50}$}&$\ast$\\
   $\varphi_v$&0&0&$\hat{k}_2$&0&$c_{12}$&$\frac12c_{03}$&
   \fbox{$\frac16\hat{c}_{31}$}&$\frac12\hat{c}_{22}$&$\frac12c_{13}$&0&
   $\frac1{24}\hat{c}_{41}$&$\ast$\\
   \hline
   $\varphi$& 0&0&0&0&0&$\frac12\hat{k}_2$&
   0&0&$\frac12c_{12}$&$\frac16c_{03}$&0&$\frac1{120}\hat{c}_{50}$
   \\
   $u\varphi_u$&0&0&0&0&0&0&0&0&$\frac12c_{12}$&0&0&\fbox{$\frac1{24}\hat{c}_{50}$}
   \\
   $v\varphi_u$&0&0&0&0&0&0&0&0&0&$\frac12c_{12}$&0&0\\
   $u\varphi_v$&0&0&0&0&\fbox{$\hat{k}_2$}&0&0&$c_{12}$&$\frac12c_{03}$
   &0&$\frac16\hat{c}_{31}$&$\frac1{24}\hat{c}_{41}$\\
   $v\varphi_v$&0&0&0&0&0&\fbox{$\hat{k}_2$}&0&0&$c_{12}$&$\frac12c_{03}$&0&0\\
   \hline
   $u^2\varphi_v$&0&0&0&0&0&0&0&\fbox{$\hat{k}_2$}&0&0&0&$\frac16\hat{c}_{31}$\\
   $uv\varphi_v$&0&0&0&0&0&0&0&0&\fbox{$\hat{k}_2$}&0&0&0\\
   $v^2\varphi_v$&0&0&0&0&0&0&0&0&0&\fbox{$\hat{k}_2$}&0&0\\
   \hline
 \end{tabular}  \vskip3mm
 \begin{tabular}{c|c|ccccc|c}
  \hline
  &$u^iv^j$ ($i+j\le3$)&$u^4$&$u^3v$&$u^2v^2$&$uv^3$&$v^4$&$u^5$\\
  \hline
  $u^3\varphi_v$&0&0&\fbox{$\hat{k}_2$}&0&0&0&0\\
  $u^2v\varphi_v$&0&0&0&\fbox{$\hat{k}_2$}&0&0&0\\
  $uv^2\varphi_v$&0&0&0&0&\fbox{$\hat{k}_2$}&0&0\\
  $v^3\varphi_v$&0&0&0&0&0&\fbox{$\hat{k}_2$}&0\\
  \hline
 \end{tabular}
 \vskip3mm 
 \begin{tabular}{c|c|cccccc}
  \hline
  &$u^iv^j$ ($i+j\le4$)&$u^5$&$u^4v$&$u^3v^2$&$u^2v^3$&$uv^4$&$v^5$\\
  \hline
  $u^4\varphi_v$&0&0&\fbox{$\hat{k}_2$}&0&0&0&0\\
  $u^3v\varphi_v$&0&0&0&\fbox{$\hat{k}_2$}&0&0&0\\
  $u^2v^2\varphi_v$&0&0&0&0&\fbox{$\hat{k}_2$}&0&0\\
  $uv^3\varphi_v$&0&0&0&0&0&\fbox{$\hat{k}_2$}&0\\
  $v^4\varphi_v$&0&0&0&0&0&0&\fbox{$\hat{k}_2$}\\
  \hline
 \end{tabular}
\end{center}
 \colB{Here, 
 \begin{align*}
  \hat{c}_{40}&=(\hat{k}_2c_{40}-3{c_{21}}^{2})/\hat{k}_2,\quad
  \hat{c}_{31}=(\hat{k}_2c_{31}-3c_{21}c_{12})/\hat{k}_2,\quad
  \hat{c}_{22}=(\hat{k}_2c_{22}-c_{21}c_{03})/\hat{k}_2,\\
  \hat{c}_{50}&=({\hat{k}_2}^2c_{50}-10\hat{k}_2c_{21}c_{31}+15{c_{21}}^{2}c_{12})/{\hat{k}_2}^{2},\quad
  \hat{c}_{41}=({\hat{k}_2}^2c_{41}-6\hat{k}_2c_{21}c_{22}+3{c_{21}}^2c_{03})/{\hat{k}_2}^2,
 \end{align*}
 and so on. The coefficients mentioned by ``$*$'' are not important.}
 \colB{The equality \eqref{eq:VerA4} holds if and only if the matrix
 presented by this table is of full rank.}  
 \colB{Using Gauss's elimination method using boxed elements as pivots,
 we conclude that $\Phi$ is $\mathcal{K}$-versal if and only if $\hat{c}_{31}\ne0$.}
% We claim that $\Phi$ is $\mathcal K$-versal when $\hat{c}_{31}\ne0$.  
 The condition $\hat{c}_{31}\ne0$ is equivalent to \colB{$3a_{12}a_{21}+a_{31}(k_1-k_2)\ne0$}
%  $$
%  3a_{12}a_{21}+a_{31}(k_1-k_2)\ne0
%  $$
 in the original coefficients of the Monge form. 
 From Lemma \ref{lem:ridgeline}, 
 $\Phi$ is $\mathcal{K}$-versal unfolding of $\varphi$ if and only
 if $(0,0)$ is a %regular point
 \colB{non-singular point} of the ridge line relative to ${\bf v}_1$. 
\end{proof}

%\subsection{$D_4^\pm$-singularity}
\begin{proof}
 [Proof of Theorem \corR{$\ref{thm:main1}$ $(4)$}]
 From Theorem \ref{thm:CriteriaD4},
 $\varphi$ is $\mathcal K$-equivalent to $D_4^+$ %(resp.~$D_4^-$)
 at $(0,0)$ if
 \[
 \hat{k}_1=\hat{k}_2=0,\text{ and }  
 \colR{\begin{vmatrix}
  c_{30}&2c_{21}&c_{12}&0\\
  0&c_{30}&2c_{21}&c_{12}\\
  c_{21}&2c_{12}&c_{03}&0\\
  0&c_{21}&2c_{12}&c_{03}
 \end{vmatrix}>0.}
 \]
 These conditions are equivalent to 
 \[
  k_1=k_2=\frac1{z_0},\text{ and }
 \begin{vmatrix}
  a_{30}&2a_{21}&a_{12}&0\\
  0&a_{30}&2a_{21}&a_{12}\\
  a_{21}&2a_{12}&a_{03}&0\\
  0&a_{21}&2a_{12}&a_{03}
 \end{vmatrix}>0 
 \]
 in the original coefficient\corC{s} of the Monge form.
 Therefore, $\varphi$ is $\mathcal{K}$-equivalent to $D_4^+$
 at $(0,0)$ if the origin is %an elliptice
 \colB{a hyperbolic umbilic} (see Section
 \ref{sec:umbilics}).
 
 \colB{We assume that $\varphi$ has a $D_4^+$-singularity at $(0,0)$.}
 Since $D_4^\pm$-singularity is 3-determined,  
 $\Phi$ is $\mathcal K$-versal unfolding of $\varphi$ if and only if 
 \begin{equation}\label{VerD4}
  \mathcal E_2=
   \langle\varphi_u,\varphi_v,\varphi\rangle_{\mathcal E_2}
   +
   \langle
   {\Phi}_x|_{\BB{R}^2\times \colB{q}},
   {\Phi}_y|_{\BB{R}^2\times \colB{q}},
   {\Phi}_z|_{\BB{R}^2\times \colB{q}},
   {\Phi}_t|_{\BB{R}^2\times \colB{q}}
   \rangle_{\BB{R}}
   \colR{+\langle u,v\rangle^4.}
 \end{equation}
 The coefficients of $u^iv^j$ of functions appearing in \eqref{VerD4}
 are given by the following tables:  
 \begin{center}
 \begin{tabular}{c|c|cc|ccc|ccccc}
  \hline
  &1&$u$&$v$&$u^2$&$uv$&$v^2$&$u^3$&$u^2v$&$uv^2$&$v^3$\\
  \hline
  $\Phi_x$&0&1&0&0&0&0&0&0&0&0\\
  $\Phi_y$&0&0&1&0&0&0&0&0&0&0\\
  $\Phi_z$&$-z_0$&0&0&$\frac12k_1$&0&$\frac12k_2$&
  $\frac16a_{30}$&$\frac12a_{21}$&$\frac12a_{\corC{12}}$&$\frac16a_{\corC{03}}$\\
  $\Phi_t$&$t_0$&0&0&0&0&0&0&0&0&0\\
  \hline
  $\Phi_u$&0&0&0&
  $\frac12c_{30}$&$c_{21}$&$\frac12c_{12}$&
  $\frac16c_{40}$&$\frac12c_{31}$&$\frac12c_{22}$&$\frac16c_{13}$\\
  $\Phi_v$&0&0&0&$\frac12c_{21}$&$c_{12}$&$\frac12c_{03}$&
  $\frac16c_{31}$&$\frac12c_{22}$&$\frac12c_{13}$&$\frac16c_{04}$
  \\
  \hline
  %$\Phi-z\Phi_z$& 0&0&0&$-\frac12$&0&$-\frac12$&0&0&0&0\\
  $u\Phi_u$&0&0&0&0&0&0&$\frac12c_{30}$&$c_{21}$&$\frac12c_{12}$&0\\ 
  $v\Phi_u$&0&0&0&0&0&0&0&$\frac12c_{30}$&$c_{21}$&$\frac12c_{12}$\\
  $u\Phi_v$&0&0&0&0&0&0&$\frac12c_{21}$&$c_{12}$&$\frac12c_{03}$&0\\
  $v\Phi_v$&0&0&0&0&0&0&0&$\frac12c_{21}$&$c_{12}$&$\frac12c_{03}$\\
  \hline
 \end{tabular}
 \end{center}
 Thus we obtain that $\Phi$ is $\mathcal K$-versal if and only if 
 $$
 \begin{vmatrix}
  1&0&1\\
  c_{30}&c_{21}&c_{12}\\
  c_{21}&c_{12}&c_{03}
 \end{vmatrix}
 \ne0.
 $$
 This condition is equivalent to
 $$
 \begin{vmatrix}
  1&0&1\\
  a_{30}&a_{21}&a_{12}\\
  a_{21}&a_{12}&a_{03}
 \end{vmatrix}
 \ne0
 $$
 in the original coefficient\corC{s} of the Monge form.
 \colB{This condition is equivalent to the origin is not a right-angled
 umbilic.} 
 Hence, we complete the proof.
%  Since the proof of Theorem~\ref{thm:main1}~(5) is completely
%  parallel to that of Theorem~\ref{thm:main1}~(4), we omit the detail.
\end{proof}

%%%%%%%%%%%% 参考文献 %%%%%%%%%%%%%
%%
% <著者名>
%%% スモールキャピタルでGiven Name頭文字、Family Nameの順に記述
% <タイトル>
%%% 書籍、論文集の場合もローマン体で記述
% <雑誌名>
%%%雑誌名の略記はMathematical Reviews誌のMR Serials Abbreviations List
%%%(Abbreviations of Names of Serials) の方法に従うこと．
% <巻>
%%% 基本的に号(Number)は省き，巻(Volume)のみを記述
% <ページ数の表記>
%%% 「--」ハイフン２ヶを使用
%%

\bigskip
%%%%%%%%%%%% 著者所属 %%%%%%%%%%%%%
\address{% 第一著者
Department of Mathematics\\ %適当な箇所で改行
Faculty of Science\\
Saitama University\\
255 Shimo-Okubo, Sakura-Ku\\
Saitama 338-8570\\
Japan
}
{tfukui@rimath.saitama-u.ac.jp}% 第一著者の後は改行せずに続ける
%%%%%%%%%
\address{% 第二著者
Department of Mathematics\\ %適当な箇所で改行
Faculty of Science\\
Saitama University\\
255 Shimo-Okubo, Sakura-Ku\\
Saitama 338-8570\\
Japan
}
{s07dm003@mail.saitama-u.ac.jp}

% %%%%%%%%%%
% \address{% 第三著者
% Mathematical Institute \\
% Tohoku University \\
% Sendai 980-8578 \\
% Japan
% }
% {author3@math.tohoku.ac.jp}% 第三著者の後は改行せずに続ける
% %%%%%%%%%
% \address{% 第四著者
% Mathematical Institute \\
% Tohoku University \\
% Sendai 980-8578 \\
% Japan
% }
% {author4@math.tohoku.ac.jp}

\end{document}